\documentclass[reqno,11pt,psamsfonts]{amsart}

\usepackage[all]{xy}
\usepackage{amsthm}
\usepackage{amssymb}
\usepackage{amsmath,amscd}
\usepackage{mathrsfs}
\usepackage[dvips]{graphicx, color}

\usepackage[dvips]{graphicx, color}

\def\a{\alpha}
\def\b{\beta}
\def\c{\gamma}
\def\d{\delta}

\def\g{\gamma}
\def\l{\lambda}

\def\NN{{\mathbb N}}
\def\PP{{\mathbb P}}

\def\RR{{\mathbb R}}

\def\ZZ{{\mathbb Z}}

\def\cal{\mathcal}

\def\cA{{\cal A}}

\def\cD{{\cal D}}

\def\cP{{\cal P}}

\def\cV{{\cal V}}

\def\Aut{\operatorname{Aut}}

\def\dim{\operatorname{dim}}

\def\Ext{\operatorname {Ext}}

\def\GL{\operatorname {GL}}

\def\GKdim{\operatorname{GKdim}}
\def\GL{\operatorname{GL}}
\def\gldim{\operatorname{gldim}}

\def\GrMod{\operatorname{GrMod}}

\def\GrAut{\operatorname{GrAut}}

\def\id{\operatorname {id}}

\def\PGL{\operatorname{PGL}}

\def\<{\langle}
\def\>{\rangle}

\def\kxy{k \langle x,y \rangle}
\def\kxyz{k \langle x,y,z \rangle}

\def\rnum#1{\expandafter{\romannumeral #1}}
\def\Rnum#1{\uppercase\expandafter{\romannumeral #1}}

\theoremstyle{plain} 
\newtheorem{theorem}{Theorem}[section]

\newtheorem{lemma}[theorem]{Lemma}

\theoremstyle{definition}
\newtheorem{definition}[theorem]{Definition}

\theoremstyle{remark}

\newtheorem{remark}[theorem]{Remark}

\numberwithin{equation}{section}

\begin{document}
	\pagenumbering{arabic}
	
	\title[Defining relations of 3-dim. cubic AS-reg alg. of Type P, S and T]
	{Defining relations of 3-dimensional cubic AS-regular algebras
	of Type P, S and T}
	
	\author{Masaki Matsuno and Yu Saito}
	
	\address{Institute of Arts and Sciences,
		Tokyo university of Science, 6-3-1 Niijuku, Katsushika-ku, Tokyo 125-8585, JAPAN}
	\email{masaki.matsuno@rs.tus.ac.jp}
	\address{Graduate School of Integrated Science and Technology,
		Shizuoka University, Ohya 836, Shizuoka 422-8529, Japan}
	\email{saito.yu.18@shizuoka.ac.jp}
	
	\keywords {Graded algebras, Artin-Schelter regular algebras, Cubic algebras, Geometric algebras}
	
	\thanks {{\it 2020 Mathematics Subject Classification}: 14A22, 16W50, 16S38, 16D90, 16E65}
	
	%\noindent \today
	
	\maketitle
	%%%%%%%%%%%%%%%%%%%%%%%%%%%%%%%%%%%%%%%%%%%%%%%%%%%%%%%%%%%%%%%%%%%%%%%%%%%%%%%%%%%%%%%%%%%%%%%%%%%%%%%%%%%%%%%%%%%%%%%%%%
	\begin{abstract} 
		Classification of AS-regular algebras is one of the most important projects in noncommutative algebraic geometry.
		Recently, Itaba and the first author gave a complete list of defining relations of $3$-dimensional quadratic AS-regular algebras
		by using the notion of geometric algebra and twisted superpotential.
		In this paper, we extend the notion of geometric algebra to cubic algebras, and give a geometric condition for
		isomorphism and graded Morita equivalence.
		One of the main results is a complete list of defining relations of $3$-dimensional cubic AS-regular algebras
		corresponding to $\PP^1 \times \PP^1$ or a union of irreducible divisors of bidegree $(1,1)$ in $\PP^1 \times \PP^1$.
		Moreover, we classify them up to isomorphism and up to graded Morita equivalence in terms of their defining relations.
	\end{abstract}
	%%%%%%%%%%%%%%%%%%%%%%%%%%%%%%%%%%%%%%%%%%%%%%%%%%%%%%%%%%%%%%%%%%%%%%%%%%%%%%%%%%%%%%%%%%%%%%%%%%%%%%%%%%%%%%
	\section{Introduction}
	In \cite{AS}, Artin and Schelter introduced a nice class of noncommutative graded algebras, called Artin-Schelter regular (AS-regular) algebras,
	which have the same homological properties as the commutative polynomial algebra.
	They proved that every $3$-dimensional AS-regular algebra finitely generated in degree $1$ has
	either $3$ generators and $3$ quadratic defining relations (the quadratic case), or $2$ generators and $2$ cubic defining relations (the cubic case).
	Although Artin and Schelter attempted to determine defining relations of all $3$-dimensional AS-regular algebras finitely generated in degree $1$,
	they gave a list of defining relations of $``$generic$"$ $3$-dimensional AS-regular algebras in
	\cite[Table 3.11]{AS} (the quadratic case) and \cite[Table 3.9]{AS} (the cubic case) by using a computer software.
	In \cite{ATV}, Artin, Tate and Van den Bergh proved that every $3$-dimensional AS-regular algebra finitely generated in degree $1$
	determines and is determined by a pair $(E,\sigma)$ where $E$ is a scheme and $\sigma$ is an automorphism of $E$.
	Moreover, a list of pairs corresponding to $``$generic$"$ $3$-dimensional AS-regular algebras was given in \cite[4.13]{ATV}.
	Their work told us that algebraic geometry is a useful tool to study even noncommutative graded algebras, and
	is considered as a starting point of the research field noncommutative algebraic geometry.
	
	Although Artin, Tate and Van den Bergh found a nice one-to-one correspondence between
	$3$-dimensional AS-regular algebras and pairs $(E,\sigma)$,
	a complete list of defining relations of $``$all$"$ $3$-dimensional AS-regular algebras
	was not appeared in the literature.
	Focusing on pairs $(E,\sigma)$, Mori \cite{M} introduced the notion of geometric algebra $\cA(E,\sigma)$.
	It is a quadratic algebra which determines and is determined by a pair $(E,\sigma)$ (see \cite[Definition 4.3]{M}).
	It follows from \cite[Proposition 4.3]{ATV} that every $3$-dimensional quadratic AS-regular algebra is a geometric algebra.
	By using (G2) condition in the definition of geometric algebra,
	we can directly calculate the defining relations of geometric algebra, so
	this notion plays an important role to determine the defining relations of $3$-dimensional quadratic AS-regular algebras.
	In fact, Itaba and the first author gave a complete list of defining relations of $3$-dimensional quadratic AS-regular algebras
	by using the notion of geometric algebra and twisted superpotential
	(see \cite{IM0}, \cite{IM}, \cite{MM} in detail).
	Our results mean that the original Artin-Schelter's project in the quadratic case is complete.
	
	The next natural project is to determine the defining relations of $3$-dimensional cubic AS-regular algebras.
	Although a complete list of superpotentials (defining relations) of all $3$-dimensional cubic AS-regular algebras which are $``$Calabi-Yau$"$ was given in \cite{MU2},
	no complete list of defining relations of $``$all$"$ $3$-dimensional cubic AS-regular algebras has appeared in the literature.
	The ultimate goal of our project is 
	\begin{enumerate}
		\item[{\rm (I)}]
		to give a complete list of defining relations of $``$all$"$ $3$-dimensional cubic AS-regular algebras,
		\item[{\rm (II)}]
		to classify them up to isomorphism in terms of their defining relations, and
		\item[{\rm (III)}]
		to classify them up to graded Morita equivalence in terms of their defining relations.
	\end{enumerate}
	As a first step, we extend the notion of geometric algebra $\cA(E,\sigma)$ to cubic algebras (see Definition \ref{defin.GA}).
	In the rest of this paper, we tacitly assume that all geometric algebras are cubic algebras.
	For geometric algebras, we give a nice geometric condition to classify them up to isomorphism of graded algebra and graded Morita equivalence
	(see Theorems \ref{thm.G1} and \ref{thm.G2}).
	It follows from \cite[Proposition 4.3]{ATV} that every $3$-dimensional cubic AS-regular algebra is a geometric algebra $\cA(E,\sigma)$.
	Moreover, $E$ is $\PP^1 \times \PP^1$ or a divisor of bidegree $(2,2)$ in $\PP^1 \times \PP^1$.
	In this paper, we completed our project in the cases that $E$ is $\PP^1 \times \PP^1$ (called Type P), or a union of irreducible divisors of bidegree $(1,1)$ in $\PP^1 \times \PP^1$
	(called Type S or Type T)
	(see Theorems \ref{thm.main1} and \ref{thm.main2}).
	We remark that when we classify $3$-dimensional cubic AS-regular algebras, we need to classify divisors of bidegree $(2,2)$ in $\PP^1 \times \PP^1$ up to equivalence (see Definition \ref{defin.2PE}).
	If two divisors are equivalent, then they are isomorphic, however, the converse is not true.
	This is a major difference from the classification of $3$-dimensional quadratic AS-regular algebras and the most difficult aspect of the cubic case.
	
	This paper is organized as follows:
	In Section 2, we recall the definition of an Artin-Schelter regular algebra from \cite{AS}, a twisted algebra from \cite{Z}, and a derivation-quotient algebra from
	\cite{BSW}, \cite{MS}, \cite{IM}.
	In Section 3, we extend the notion of geometric algebra to cubic algebras (see Definition \ref{defin.GA}), and
	give a geometric condition to classify geometric algebras up to isomorphism of graded algebra and graded Morita equivalence (see Theorems \ref{thm.G1} and \ref{thm.G2}).
	We also show that a twisted algebra of a geometric algebra is also geometric (see Theorem \ref{thm.GM}).
	Finally, in Section 4, we give a complete list of defining relations of $3$-dimensional cubic AS-regular algebras $\cA(E,\sigma)$ where $E$ is $\PP^1 \times \PP^1$ or
	a union of irreducible divisors of bidegree $(1,1)$ in $\PP^1 \times \PP^1$, and conditions for isomorphism and graded Morita equivalence in terms of their defining relations
	(see Theorems \ref{thm.main1} and \ref{thm.main2}).
	
	%%%%%%%%%%%%%%%%%%%%%%%%%%%%%%%%%%%%%%%%%%%%%%%%%%%%%%%%%%%%%%%%%%%%%%%%%%%%%%%%%%%%%%%%%%%%%%%%%%%%%%%%%%%%%%%%%%%%%%%%%%%%%%%%%%%%%%%%%%%%%%%%%%%%%%%%%%%%%%%%%%%%%%%%%%%%%%%
	\section{Preliminary}
	Throughout this paper, we fix an algebraically closed field $k$ of characteristic zero, and assume that
	all algebras, vector spaces and unadorned tensor products are over $k$.
	A graded algebra means an $\NN$-graded algebra $A=\oplus_{i \in \NN}A_i$.
	A \textit{connected} graded algebra $A$ is a graded algebra such that $A_{0}=k$.
	
	Let $V$ be a finite dimensional vector space.
	We denote by $T(V)$ the \textit{tensor algebra} of $V$ over $k$.
	By setting $T(V)_i=V^{\otimes i}$, $T(V)$ becomes an $\NN$-graded algebra.
	A \textit{cubic algebra} is an $\NN$-graded algebra of the form $T(V)/(R)$
	where $R$ is a subspace of $V^{\otimes 3}$.
	If a connected graded algebra $A$ is finitely generated in degree $1$ over $k$, then $A$ is isomorphic to a graded quotient algebra $T(A_1)/I$
	where $I$ is an two-sided ideal generated by homogeneous elements of $T(A_1)$.
	We denote by $\GrAut_k\,A$ the group of graded algebra automorphisms of $A$.
	
	Let $A$ be a graded algebra. We denote by $\GrMod A$ the category of graded right $A$-modules and 
	graded right $A$-module homomorphisms preserving degrees.
	We say that two graded algebras $A$ and $A'$ are \textit{graded Morita equivalent} if the categories $\GrMod A$ and $\GrMod A'$ are equivalent.
	If $A$ is connected, then we view $k=A/A_{\ge 1}$ as a graded right $A$-module. 
	
	%%%%%%%%%%%%%%%%%%%%%%%%%%%%%%%%%%%%%%%%%%%%%%%%%%%%%%%%%%%%%%%%%%%%%%%%%%%%%%%%%%%%%%%%%%%%%%%%%%%%%%%%%%%%%
	\subsection{Artin-Schelter regular algebras}
	Let $A$ be a graded algebra.
	We recall that
	\begin{center}
		$\GKdim A:=\inf\{\a \in \RR \mid \dim(\sum_{i=0}^{n}A_i) \leq n^{\a} \,\,\,\text{for all}\,\,\,n \gg 0\}$
	\end{center}
	is called the \textit{Gelfand-Kirillov dimension} of $A$.
	In noncommutative algebraic geometry,
	Artin-Schelter regular algebras are main objects to study.
	
	\begin{definition}[{\cite{AS}}]\label{def.AS}
		A connected graded algebra $A$ is called a $d$-dimensional Artin-Schelter regular (simply AS-regular) algebra
		if $A$ satisfies the following conditions:
		\begin{enumerate}
			\item[{\rm (1)}]
			$\gldim A=d < \infty$,
			\item[{\rm (2)}]
			$\GKdim A<\infty$,
			\item[{\rm (3)}]
			$\Ext^{i}_{A}(k,A)=
			\begin{cases}
				k \quad\quad(i=d), \\
				0 \quad\quad(i \neq d).
			\end{cases}$
		\end{enumerate}
	\end{definition}
	
	It follows from \cite[Theorem 1.5 (i)]{AS} that
	a $3$-dimensional AS-regular algebra $A$ finitely generated in degree $1$ over $k$ is
	one of the following forms:
	\begin{center}
		$A=\kxyz/(f_1,f_2,f_3)$
	\end{center}
	where $f_i$ are homogeneous elements of degree $2$ (quadratic case), or
	\begin{center}
		$A=\kxy/(g_1,g_2)$
	\end{center}
	where $g_j$ are homogeneous elements of degree $3$ (cubic case).
	In this paper, we focus on studying $3$-dimensional cubic AS-regular algebras.
	
	%Let $V$ be a $2$-dimensional vector space and fix a basis $\{x_1,x_2\}$ for $V$.
	Let $T(V)/(R)$ be a cubic algebra such that $\dim V=\dim R=2$.
	We choose a basis $\{x_1,x_2\}$ for $V$ and a basis $\{f_1,f_2\}$ for $R$.
	We may write $f_i=m_{i1}x_1+m_{i2}x_2$ for some $2 \times 2$ matrix
	$\bf{M}=
	\begin{pmatrix}
		m_{11} & m_{12} \\
		m_{21} & m_{22}
	\end{pmatrix}$ 
	whose entries are in $V \otimes V$.
	We define $g_i=x_1m_{1i}+x_{2}m_{2i}$.
	If we write ${\bf x}=(x_1,x_2)^{t}$, ${\bf f}=(f_1,f_2)^{t}$ and ${\bf g}=(g_1,g_2)$, then
	${\bf f}={\bf Mx}$ and ${\bf g}={\bf x}^{t}{\bf M}$.
	%We say that $A$ is \textit{standard} if there exists $Q \in \GL_2(k)$ such that ${\bf g}^{t}=Q{\bf f}$.
	We say that
	$T(V)/(R)$ is \textit{standard} if there exists a choice of a basis $\{x_1,x_2\}$ for $V$
	and a basis $\{f_1,f_2\}$ for $R$ such that ${\bf g}^{t}=Q{\bf f}$ for some $Q \in \GL_2(k)$.
	%\begin{definition}[{\cite{ATV}}]
	%	Let $A=T(V)/(R)$ be a cubic algebra such that $\dim V=\dim R=2$.
	%	We say that
	%	$A$ is \textit{standard} if there exists a choice of a basis $\{x_1,x_2\}$ for $V$
	%	and a basis $\{f_1,f_2\}$ for $R$ such that ${\bf g}^{t}=Q{\bf f}$ for some $Q \in \GL_2(k)$.
	%\end{definition}
	Viewing ${\bf M} \in M_2(k[x_1,x_2] \circ k[x_1,x_2])$,
	we have the common zero set of the entries of ${\bf M}$ in $\PP^1 \times \PP^1$
	%we define the \textit{noncommutative Hessian}, denoted by $H({\bf M})$, by
	%$H({\bf M})=\det{\bf M} \in k[x_1,x_2] \circ k[x_1,x_2]$
	where $k[x_1,x_2] \circ k[x_1,x_2]$ denotes the Segre product of the copies of $k[x_1,x_2]$.
	\begin{theorem}[{\cite[Theorem 1]{ATV}}]\label{thm.ATV}
		Let $V$ be a $2$-dimensional vector space and $R$ a $2$-dimensional subspace of $V^{\otimes 3}$.
		Then $T(V)/(R)$ is a $3$-dimensional AS-regular algebra if and only if
		it is standard and the common zero set of the entries of ${\bf M}$ is empty in $\PP^1 \times \PP^1$.
	\end{theorem}
	
	%%%%%%%%%%%%%%%%%%%%%%%%%%%%%%%%%%%%%%%%%%%%%%%%%%%%%%%%%%%%%%%%%%%%%%%%%%%%%%%%%%%%%%%%%%%%%%
	\subsection{Twisted algebras}
	We recall the notion of twisting system and twisted algebra introduced by Zhang \cite{Z}.
	
	\begin{definition}[{\cite{Z}}]
		Let $A$ be a graded algebra.
		A set of graded linear automorphisms $\theta=\{\theta_n\}_{n \in \ZZ}$ of $A$ is called
		a \textit{twisting system} on $A$ if
		\begin{center}
			$\theta_n(a\theta_m(b))=\theta_n(a)\theta_{n+m}(b)$
		\end{center}
		for any $l,m,n \in \ZZ$ and $a \in A_m, b \in A_l$.
		The \textit{twisted algebra} of $A$ by a twisting system $\theta$, denoted by $A^{\theta}$,
		is a graded algebra $A$ with a new multiplication $\ast$ defined by
%		\begin{center}
			$a \ast b=a \theta_m(b)$
%		\end{center}
		for any $a \in A_m, b \in A_l$.
	\end{definition}
	
	Let $A$ be a graded algebra. Any graded algebra automorphism $\theta \in \GrAut_k\,A$ defines a twisting system of $A$ by $\{\theta^n\}_{n \in \ZZ}$.
	The twisted algebra of $A$ by $\{\theta^n\}_{n \in \ZZ}$ is denoted by $A^{\theta}$ instead of $A^{\{\theta^n\}_{n \in \ZZ}}$.
	The following lemma is useful to construct twisting systems.
	
	\begin{lemma}[{\cite[Proposition 2.8]{Z}}]\label{lem.Zhang}
		A graded algebra $A'$ is isomorphic to a twisted algebra of a graded algebra $A$
		if and only if there exists a set of graded linear isomorphisms $\{\phi_n\}_{n \in \ZZ}$ from $A'$ to $A$
		which satisfy
		\begin{center}
			$\phi_{n}(ab)=\phi_{n}(a)\phi_{n+m}(b)$
		\end{center}
		for any $l,m,n \in \ZZ$ and $a \in A'_m, b \in A'_l$.
		If this is the case, then setting $\theta_n=\phi_{n} \circ \phi_0^{-1}$ for any $n \in \ZZ$,
		$\theta=\{\theta_n\}_{n \in \ZZ}$ is a twisting system on $A$ and $\phi_0:A' \to A^{\theta}$ is a
		graded algebra isomorphism.
	\end{lemma}
	
	Zhang found a necessary and sufficient condition for graded Morita equivalence.
	
	\begin{theorem}[{\cite[Theorem 3.5]{Z}}]\label{thm.Zhang}
		Let $A$ and $A'$ be connected graded algebras with $A_1 \neq 0$.
		Then $A'$ is isomorphic to a twisted algebra of $A$ if and only if
		$\GrMod A$ is equivalent to $\GrMod A'$.
	\end{theorem}
	
	%%%%%%%%%%%%%%%%%%%%%%%%%%%%%%%%%%%%%%%%%%%%%%%%%%%%%%%%%%%%%%%%%%%%%%%%%%%%%%%%%%%%%%%%%%%%%%%
	\subsection{Derivation-quotient algebras}
	We now recall the definitions of twisted superpotentials and derivation-quotient algebras
	from \cite{BSW}, \cite{IM} and \cite{MS}.
	
	Let $V$ be a $2$-dimensional vector space.
	We define the linear map $\varphi: V^{\otimes 4} \to V^{\otimes 4}$ by
	$\varphi(v_1 \otimes v_2 \otimes v_3 \otimes v_4):=v_4 \otimes v_1 \otimes v_2 \otimes v_3$.
	We denote by $\GL(V)$ the general linear group of $V$.
	
	\begin{definition}
		[{\cite[Introduction]{BSW}, \cite[Definition 2.5]{MS}, \cite[Definition 2.9]{IM}}]
		%Let $w \in V^{\otimes 4}$.
		\noindent
		\begin{enumerate}
			\item[{\rm (1)}] We say that $w \in V^{\otimes 4}$ is a {\it superpotential}
			if $\varphi(w)=w$.
			%\item[{\rm (2)}] We say that $w \in V^{\otimes 4}$ is an {\it anti-superpotential}
			%if $\varphi(w)=-w$.
			\item[{\rm (2)}] We say that $w \in V^{\otimes 4}$ is a {\it twisted superpotential}
			if there exists $\theta \in \GL(V)$ such that
			$(\theta \otimes \id^{\otimes 3})(\varphi(w))=w$.
			\item[{\rm (3)}] For a superpotential $w \in V^{\otimes 4}$ and $\theta \in \GL(V)$,
			\begin{center}
				$w^{\theta}:=(\theta^3 \otimes \theta^2 \otimes \theta \otimes \id)(w)$
			\end{center}
			is called a {\it Mori-Smith twist (MS twist)} of $w$ by $\theta$.
		\end{enumerate}
	\end{definition}
	
	For $w \in V^{\otimes 4}$, we define
	\begin{center}
		$\Aut(w):=\{
		\theta \in \GL(V) \mid (\theta^{\otimes 4})(w)=\l w,\ \exists \l \in k \setminus \{0\}
		\}$.
	\end{center}
	By \cite[Lemma 3.1]{MS}, $\Aut(w)$ is a subset of $\GrAut_k \cD(w)$.
	For $\theta \in \Aut(w)$,
	we also denote by $\theta$ the graded algebra automorphism of $A$ induced by $\theta$.
	
	\begin{lemma}\label{lem.MStwist}
		Let $w \in V^{\otimes 4}$ be a superpotential
		and $\theta \in \Aut(w)$.
		Then the MS twist $w^{\theta}$ of $w$ by $\theta$ is a twisted superpotential.
	\end{lemma}
	
	\begin{proof}
		Let $w \in V^{\otimes 4}$ and $\theta \in \Aut(w)$.
		By definition, there exists $\l \in k \setminus \{0\}$ such that
		$(\theta^{\otimes 4})(w)=\l w$.
		If $w$ is a superpotential, then we set $\theta':=\l^{-1}\theta^4$.
		Since $w$ is a superpotential, it follows that
		\begin{align*}
			(\theta' \otimes \id^{\otimes 3})(\varphi(w^{\theta}))
			&=(\theta' \otimes \id^{\otimes 3})
			(\id \otimes \theta^3 \otimes \theta^2 \otimes \theta)(\varphi(w)) \\
			&=(\theta' \otimes \theta^3 \otimes \theta^2 \otimes \theta)(w)
			=w^{\theta},
		\end{align*}
		so $w^{\theta}$ is a twisted superpotential.
	\end{proof}

	Fix a basis $\{x,y\}$ for $V$.
	For $w \in V^{\otimes 4}$, there exist unique $w_x, w_y \in V^{\otimes 3}$ such that
	$w=x \otimes w_x+y \otimes w_y$.
	Then the {\it left partial derivative} of $w$ with respect to $x$ (resp. $y$) is
	$\partial_x(w):=w_x$ (resp. $\partial_y(w):=w_y$), and
	the {\it derivation-quoitent algebra} of $w$ is
	\begin{center}
		$\cD(w):=k\langle x,y \rangle/(\partial_x(w), \partial_y(w))$.
	\end{center}
	
	\begin{remark}
		By Dubois-Violette \cite{D}, Bocklandt {\it et al} \cite{BSW} and Mori-Smith \cite{MS},
		for any $3$-dimensional cubic AS-regular algebra $A$, there exists
		a unique twisted superpotential $w$ up to nonzero scalar multiples such that
		$A$ is a derivation-quotient algebra $\cD(w)$ of $w$.
		This result means that
		a twisted superpotential is a useful tool to classify $3$-dimensional cubic AS-regular algebra.
	\end{remark}

%	Then a sequence $\{\theta^n\}_{n \in \ZZ}$ is a twisting system of $\cD(w)$
%	and we also denote by $\theta$.
	
	\begin{lemma}[{\cite[Proposition 5.2]{MS}}]\label{lem.dq}
		Let $w \in V^{\otimes 4}$ and $\theta \in \Aut(w)$. Then
		$\cD(w^{\theta}) \cong \cD(w)^{\theta}$.
	\end{lemma}
	
	\begin{lemma}\label{lem.tsp}
		Let $w \in V^{\otimes 4}$ be a twisted superpotential. 
		If $\partial_x(w), \partial_y(w)$ are linearly independent,
		then the derivation-quotient algebra $\cD(w)$ is standard.
	\end{lemma}
	
	\begin{proof}
		Let $w \in V^{\otimes 4}$ be a twisted superpotential,
		that is, there exists $\theta \in \GL_2(k)$ such that
		$(\theta \otimes \id^{\otimes 3})(\varphi(w))=w$.
		Assume that $\partial_x(w), \partial_y(w)$ are linearly independent,
		so $\dim(\partial_x(w),\partial_y(w))=2$.  
		Here, there exist unique $w_x, w_y, w'_x, w'_y \in V^{\otimes 3}$ such that
		\begin{center}
			$w=x \otimes w_x+y \otimes w_y=w'_x \otimes x+w'_y \otimes y$.
		\end{center}
		If we write ${\bf f}=(w_x,w_y)^{t}$ and ${\bf g}=(w'_x,w'_y)$, then we have that ${\bf g}^{t}=(\theta^{-1})^{t}{\bf f}$.
		Therefore, $\cD(w)$ is standard.
	\end{proof}
	%%%%%%%%%%%%%%%%%%%%%%%%%%%%%%%%%%%%%%%%%%%%%%%%%%%%%%%%%%%%%%%%%%%%%%%%%%%%%%%%%%%%%%%%%%%%%%%%%%%%%%%%%%%%%%%%%%%%%%%%%%%%%
	\section{Geometric algebras}\label{sec.GA}
	In this section,
	we introduce the notion of geometric algebra
	and show useful results to classify geometric algebras up to graded algebra isomorphism and graded Morita equivalence.
	%Let $T=k \langle x_1,\dots, x_n \rangle$ denote the free associative $k$-algebra on generators $x_1, \dots, x_n$
	%of degree $1$.
	%We assume that all unadorned tensor products are over $k$.
	
	Let $V$ be a finite dimensional vector space over $k$.
	%We set the following equivalence relation on $V \setminus \{0\}$:
	%\begin{center}
	%	$u \sim v \Longleftrightarrow$ there exists $\l \in k \setminus \{0\}$ with $u=\l v$.
	%\end{center}
	%The \textit{projective space} associated to $V$ is defined by
	%\begin{center}
	%	$\PP(V):=V \setminus \{0\}/\sim$.
	%\end{center}
	We denote by $\PP(V)$ the projective space associated to $V$.
	For $v \in V \setminus \{0\}$, we denote by $\overline{v}$ the element of the projective space $\PP(V)$.
	We denote by $V^{\ast}$ the dual space of $V$.
	For $\phi \in \GL(V)$,
	we denote by $\phi^{\ast}$ the dual map of $\phi$ from $V^{\ast}$ to $V^{\ast}$.
	We define a map $\overline{\phi^{\ast}}:\PP(V^{\ast}) \to \PP(V^{\ast})$ by
	\begin{center}
		$\overline{\phi^{\ast}}\left(\overline{\xi}\right):=\overline{\phi^{\ast}(\xi)}$.
	\end{center}
	%Note that the map $\overline{\phi^{\ast}}$ is an automorphism of the projective space $\PP(V^{\ast})$.
	%For $\phi_1,\phi_2,\phi_3 \in \GL(V)$, we define a map
	%$\phi_1 \times \phi_2 \times \phi_3:V^{\times 3} \to V^{\otimes 3}$ by
	%\begin{center}
	%	$(\phi_1 \times \phi_2 \times \phi_3)(v_1,v_2,v_3):=\phi_1(v_1) \otimes \phi_2(v_2) \otimes \phi_3(v_3)$.
	%\end{center}
	%Since the map $\phi_1 \times \phi_2 \times \phi_3$ is a multilinear map, it induces a linear map
	%$\phi_1 \otimes \phi_2 \otimes \phi_3: V^{\otimes 3} \to V^{\otimes 3}$ by
	%\begin{center}
	%	$(\phi_1 \otimes \phi_2 \otimes \phi_3)(v_1 \otimes v_2 \otimes v_3)=
	%	\phi_1(v_1) \otimes \phi_2(v_2) \otimes \phi_3(v_3)$.
	%\end{center}
	%For an element $f=\sum_{i=1}^{m}u_i \otimes v_i \otimes w_i \in V^{\otimes 3}$ and elements
	%$p_1=\overline{\xi_1}, p_2=\overline{\xi_2}, p_3=\overline{\xi_3} \in \PP(V^{\ast})$, we write
	%\begin{center}
	%	$f(p_1,p_2,p_3)=\sum_{i=1}^{m}\xi_1(u_i)\xi_2(v_i)\xi_3(w_i) \in k$.
	%\end{center}
	It is easy to check the following lemma.
	\begin{lemma}\label{lem.zero}
		Let $\phi_1,\phi_2,\phi_3 \in \GL(V)$, $f \in V^{\otimes 3}$ and $p_1,p_2,p_3\in \PP(V^{\ast})$.
		Then $((\phi_1 \otimes \phi_2 \otimes \phi_3)(f))(p_1,p_2,p_3)=0$ if and only if
		\begin{center}
			$f(\overline{\phi_1^{\ast}}(p_1),\overline{\phi_2^{\ast}}(p_2),\overline{\phi_3^{\ast}}(p_3))=0$.
		\end{center}
	\end{lemma}
	
	For a subspace $R$ of $V^{\otimes 3}$, we define the zero set of $R$ by
	\begin{center}
		$\cV(R):=\{
		(p_1,p_2,p_3) \in \PP(V^{\ast})^{\times 3} \mid f(p_1,p_2,p_3)=0\,\,\, \forall f \in R
		\}$.
	\end{center}
	%Note that the zero set $\cV(R)$ is well-defined.
	
	\begin{lemma}\label{lem.zeroR}
		Let $R$ and $R'$ be subspaces of $V^{\otimes 3}$ and $\phi_1,\phi_2,\phi_3 \in \GL(V)$.
		If $(\phi_1 \otimes \phi_2 \otimes \phi_3)(R)=R'$, then
		$\cV(R)=(\overline{\phi_1^{\ast}} \times \overline{\phi_2^{\ast}} \times \overline{\phi_3^{\ast}})(\cV(R'))$.
	\end{lemma}
	
	\begin{proof}
		Assume that $(\phi_1 \otimes \phi_2 \otimes \phi_3)(R)=R'$.
		By Lemma \ref{lem.zero}, it holds that
		\begin{align*}
			&(p_1,p_2,p_3) \in \cV(R) \\
			\Longleftrightarrow\,\,\, &f(p_1,p_2,p_3)=0 \quad \forall f \in R \\
			\Longleftrightarrow\,\,\, &((\phi_1 \otimes \phi_2 \otimes \phi_3)(f))((\overline{\phi_1^{\ast}})^{-1}(p_1),(\overline{\phi_2^{\ast}})^{-1}(p_2),(\overline{\phi_3^{\ast}})^{-1}(p_3))=0 \quad \forall f \in R \\
			%\Longleftrightarrow\,\,\, &f'((\overline{\phi_1^{\ast}})^{-1}(p_1),(\overline{\phi_2^{\ast}})^{-1}(p_2),(\overline{\phi_3^{\ast}})^{-1}(p_3))=0 \quad \forall f' \in R' \\
			\Longleftrightarrow\,\,\, &((\overline{\phi_1^{\ast}})^{-1}(p_1),(\overline{\phi_2^{\ast}})^{-1}(p_2),(\overline{\phi_3^{\ast}})^{-1}(p_3)) \in \cV(R') \\
			\Longleftrightarrow\,\,\, &(p_1,p_2,p_3) \in (\overline{\phi_1^{\ast}} \times \overline{\phi_2^{\ast}} \times \overline{\phi_3^{\ast}})(\cV(R')).
		\end{align*}
		Hence, the statement holds. 
		%    Take $(p_1,p_2,p_3) \in \Gamma_R$ and $f \in R'$.
		%    Then there exists $g \in R$ such that $f=(\phi_1 \otimes \phi_2 \otimes \phi_3)(g)$.
		%    Since $g(p_1,p_2,p_3)=0$, it follows from Lemma \ref{lem.zero} that
		%    \begin{align*}
			%    	&f((\overline{\phi_1^{\ast}})^{-1}(p_1),(\overline{\phi_2^{\ast}})^{-1}(p_2),(\overline{\phi_3^{\ast}})^{-1}(p_3)) \\
			%    	&=((\phi_1 \otimes \phi_2 \otimes \phi_3)(g))
			%    	((\overline{\phi_1^{\ast}})^{-1}(p_1),(\overline{\phi_2^{\ast}})^{-1}(p_2),(\overline{\phi_3^{\ast}})^{-1}(p_3)) \\
			%    	&=0,
			%    \end{align*}
		%    so $((\overline{\phi_1^{\ast}})^{-1}(p_1),(\overline{\phi_2^{\ast}})^{-1}(p_2),(\overline{\phi_3^{\ast}})^{-1}(p_3)) \in \Gamma_{R'}$.
		%    Thus we have that $(p_1,p_2,p_3) \in (\overline{\phi_1^{\ast}} \times \overline{\phi_2^{\ast}} \times \overline{\phi_3^{\ast}})(\Gamma_{R'})$.
		%    
		%    Conversely, take $(q_1,q_2,q_3) \in (\overline{\phi_1^{\ast}} \times \overline{\phi_2^{\ast}} \times \overline{\phi_3^{\ast}})(\Gamma_{R'})$ and $h \in R$.
		%    Then there exists $(p_1,p_2,p_3) \in \Gamma_{R'}$ such that %$(q_1,q_2,q_3)=(\overline{\phi_1^{\ast}}(p_1),\overline{\phi_2^{\ast}}(p_2),\overline{\phi_3^{\ast}}(p_3))$.
		%    Since $(\phi_1 \otimes \phi_2 \otimes \phi_3)(h) \in R'$, we have that $((\phi_1 \otimes \phi_2 \otimes \phi_3)(h))(p_1,p_2,p_3)=0$.
		%    By Lemma \ref{lem.zero}, it holds that
		%    \begin{center}
			%    	$h(q_1,q_2,q_3)=h(\overline{\phi_1^{\ast}}(p_1),\overline{\phi_2^{\ast}}(p_2),\overline{\phi_3^{\ast}}(p_3))=0$.
			%    \end{center}
		%    Thus we have that $(q_1,q_2,q_3) \in \Gamma_R$. Hence, the statement holds.
	\end{proof}
	%Remark that under the above notation it follows that
	%$((\phi_1 \otimes \phi_2 \otimes \phi_3)(f))(p_1,p_2,p_3)=0$ if and only if
	%$f(\overline{\phi_1^{\ast}}(p_1),\overline{\phi_2^{\ast}}(p_2),\overline{\phi_3^{\ast}}(p_3))=0$.
	
	%Let $V$ be a finite dimensional vector space over $k$.
	%We denote by $T(V)$ the \textit{tensor algebra} of $V$ over $k$.
	%By setting $T(V)_i=V^{\otimes i}$, $T(V)$ becomes an $\NN$-graded algebra.
	%A \textit{cubic algebra} is an $\NN$-graded algebra of the form $T(V)/(R)$
	%where $R$ is a subspace of $V^{\otimes 3}$.
	%We define the notion of geometric algebra.
	Let $V$ be a finite-dimensional vector space.
	Let $E \subset \PP(V^{\ast}) \times \PP(V^{\ast})$ be a projective variety and
	$\pi_i:\PP(V^{\ast}) \times \PP(V^{\ast}) \to \PP(V^{\ast})$ $i$-th projections where $i=1,2$.
	We set the following notation:
	\begin{center}
		$\Aut_k^G E:=\{\sigma \in \Aut_k E \mid (\pi_1 \circ \sigma)(p_1,p_2)=\pi_2(p_1,p_2) \,\,\, \forall (p_1,p_2) \in E \}$.
	\end{center}
	We say that a pair $(E,\sigma)$ is a {\it geometric pair}
	if $\sigma \in \Aut_k^G E$.
	%if it consists of
	%a projective variety $E \subset \PP(V^{\ast}) \times \PP(V^{\ast})$ and an automorphism $\sigma$ of $E$
	%which satisfies 
	%\begin{center}
	%$\sigma(p_1, p_2)=(p_2, \sigma_M(p_1, p_{2}))$
	%\end{center}
	%for some morphism $\sigma_M: E \to \PP(V^{\ast})$.
	
	Let $A=T(V)/(R)$ be a cubic algebra where $V$ is a finite-dimensional vector space
	and $R$ is a subspace of $V^{\otimes 3}$.
	We denote by $\PP_i$ the $i$-th factor of the product $\PP(V^{\ast})^{\times 3}$ and let $\pi_{ij}$ denote
	its projection to the product $\PP_i \times \PP_j$.
	\begin{definition}[{cf. \cite[Definition 4.3]{M}}]\label{defin.GA}
		Let $A=T(V)/(R)$ be a cubic algebra.
		%Let $\PP_i$ denote the $i$-th factor of the product $\PP(V^{\ast})^{\times 3}$ and $\pi_{ij}$
		%its projection to the product $\PP_i \times \PP_j$.
		\begin{enumerate}
			\item
			We say that $A$ satisfies (G1) if
			$\pi_{12}$ and $\pi_{23}$ send $\cV(R)$ isomorphically onto their images $\pi_{12}(\cV(R))$ and $\pi_{23}(\cV(R))$ respectively, and
			\begin{center}
				$\pi_{12}(\cV(R))=\pi_{23}(\cV(R))$.
			\end{center}
			In this case, we write $\cP(A)=(E,\sigma)$ where $E:=\pi_{12}(\cV(R))$ and $\sigma:=\pi_{23}{\pi_{12}}^{-1} \in \Aut_k^G E$.
			%We say that $A$ satisfies (G1) if there exists a geometric pair $(E,\sigma)$ such that
			%\begin{center}
			%	$\Gamma_R=\{
			%	(p_1,p_2,\sigma_M(p_1,p_2)) \in \PP(V^{\ast})^{\times 3} \mid (p_1,p_2) \in E
			%	\}$.
			%\end{center}
			%In this case, we write $\cP(A)=(E,\sigma)$ and $E$ is called the \textit{point variety} of $A$. 
			
			\item We say that $A$ satisfies (G2) if there exists a geometric pair $(E,\sigma)$ such that
			\begin{center}
				$R=\{
				f \in V^{\otimes 3} \mid f(p_1,p_2,(\pi_2 \circ \sigma)(p_1,p_2))=0 \,\,\, \forall (p_1,p_2) \in E
				\}$.
			\end{center}
			In this case, we write $A=\cA(E,\sigma)$.
			
			\item We say that $A$ is a {\it geometric algebra} if $A$ satisfies (G1) and (G2) with $A=\cA(\cP(A))$.
		\end{enumerate}
	\end{definition}

	%Using the notion of geometric algebra, we have the important result.
	
	%\begin{theorem}[{\cite{ATV}}]\label{thm.ATV2}
	%	Every $3$-dimensional cubic AS-regular algebra $A$ is geometric.
	%	Moreover, the point variety $E$ of $A$ is either $\PP^1 \times \PP^1$ or a divisor of bidegree $(2,2)$ in $\PP^1 \times \PP^1$.
	%\end{theorem}
	
	%The notion of geometric algebra is preserved under graded Morita equivalence.
	Let $A=T(V)/(R)$ be a cubic algebra. If a graded algebra $A'=T(V')/(R')$ is graded Morita equivalent to $A$,
	then it follows from Theorem \ref{thm.Zhang} that $A'$ is isomorphic to a twisted algebra $A^{\theta}$ of $A$ for some twisting system $\theta$.
	As is explained in \cite[page 305--306]{Z}, $A^{\theta}$ is also a cubic algebra.
	Moreover, $A$ and $A^{\theta}$ have the same set of generators, so we may assume that $V=V'$.
	
	\begin{theorem}\label{thm.GM}
		Let $A=T(V)/(R)$ and $A'=T(V)/(R')$ be cubic algebras such that $\GrMod A \cong \GrMod A'$.
		Then the following statements hold:
		\begin{enumerate}
			\item[{\rm (1)}] $A$ satisfies {\rm (G1)} if and only if $A'$ satisfies {\rm (G1)}.
			\item[{\rm (2)}] $A$ satisfies {\rm (G2)} if and only if $A'$ satisfies {\rm (G2)}.
			\item[{\rm (3)}] $A$ is geometric if and only if $A'$ is geometric.
		\end{enumerate}
	\end{theorem}
	
	\begin{proof}
		In each case, it is enough to show one direction.
		
		(1) Assume that $A$ satisfies {\rm (G1)} with $\cP(A)=(E,\sigma)$.
		By Theorem \ref{thm.Zhang}, $A'$ is isomorphic to a twisted algebra of $A$.
		By Lemma \ref{lem.Zhang}, there exists a set of graded linear isomorphisms $\{\phi_{n}\}_{n \in \ZZ}$ from $A'$ to $A$ which satisfy
		\begin{center}
			$\phi_{n}(ab)=\phi_{n}(a)\phi_{n+m}(b)$
		\end{center}
		for any $l,m,n \in \ZZ$ and $a \in A'_m$, $b \in A'_l$.
		For any $n \in \ZZ$, we denote by $\phi_{n}|_V \in \GL(V)$ the restricted automorphism of $\phi_{n}$ and
		define by $\tau_n:=\overline{\phi_n|_V^{\ast}}$ an automorphism of $\PP(V^{\ast})$.
		Since $\sum_{i}u_i\otimes v_i \otimes w_i \in R'$ if and only if, for each $n \in \ZZ$,
		$\sum_{i}\phi_n|_V(u_i) \otimes \phi_{n+1}|_V(v_i) \otimes \phi_{n+2}|_V(w_i) \in R$, it holds that
		\begin{center}
			$(\phi_{n}|_V \otimes \phi_{n+1}|_V \otimes \phi_{n+2}|_V)(R')=R$
		\end{center}
		for any $n \in \ZZ$. By Lemma \ref{lem.zeroR}, we have that
		\begin{center}
			$\cV(R')=(\tau_n \times \tau_{n+1} \times \tau_{n+2})(\cV(R))$
		\end{center}
		for any $n \in \ZZ$.
		Consequently, there are commutative diagrams
		\begin{center}
			$\xymatrix@C=40pt{
				\cV(R)\ar[r]^-{\tau_n \times \tau_{n+1} \times \tau_{n+2}}_{\cong}\ar[d]_-{\pi_{12}}&\cV(R')\ar[d]^-{\pi'_{12}}\\
				E\ar[r]^-{\cong}_-{\tau_n \times \tau_{n+1}}&\pi'_{12}(\cV(R')),
			}$
			$\xymatrix@C=40pt{
				\cV(R)\ar[r]^-{\tau_n \times \tau_{n+1} \times \tau_{n+2}}_-{\cong}\ar[d]_-{\pi_{23}}&\cV(R')\ar[d]^-{\pi'_{23}}\\
				E\ar[r]^-{\cong}_-{\tau_{n+1} \times \tau_{n+2}}&\pi'_{23}(\cV(R'))
			}$
		\end{center}
		for any $n \in \ZZ$.
		Since $\cP(A)=(E,\sigma)$, $\pi'_{12}$ and $\pi'_{23}$ are isomorphisms, and
		\begin{align*}
			\pi'_{12}(\cV(R'))&=\pi'_{12}((\tau_n \times \tau_{n+1} \times \tau_{n+2})(\cV(R))) \\
			&=(\tau_n \times \tau_{n+1})(E) \\
			&=\pi'_{23}((\tau_{n-1} \times \tau_n \times \tau_{n+1})(\cV(R))) \\
			&=\pi'_{23}(\cV(R')).
		\end{align*}
		%	Define $E':=(\tau_0 \times \tau_{1})(E)$ and $\sigma':=(\tau_{1} \times \tau_2) \circ \sigma \circ (\tau_0 \times \tau_1)^{-1}$.
		%	Then $\sigma'$ is an automorphism of $E'$ and
		%	for any $(q_1,q_2) \in E'$,
		%	\begin{align*}
			%		\sigma'(q_1,q_2)&=((\tau_1 \times \tau_2)\circ \sigma)(\tau_0^{-1}(q_1),\tau_1^{-1}(q_2)) \\
			%		&=(\tau_{1} \times \tau_2)(\tau_1^{-1}(q_2),\sigma_M(\tau_0^{-1}(q_1),\tau_1^{-1}(q_2))) \\
			%		&=(q_2,\tau_2(\sigma_M(\tau_0^{-1}(q_1),\tau_1^{-1}(q_2)))) \\
			%		&=(q_2,(\tau_2 \circ \sigma_M \circ (\tau_0 \times \tau_{1})^{-1})(q_1,q_2)).
			%	\end{align*}
		%	Since $\sigma'_M:=\tau_2 \circ \sigma_M \circ (\tau_0 \times \tau_{1})^{-1}$ is a morphism from $E'$ to $\PP(V^{\ast})$,
		%	the pair $(E',\sigma')$ is a geometric pair.
		%	Since $\cP(A)=(E,\sigma)$, we have that
		%	\begin{align*}
			%		\Gamma_{R'}&=(\tau_0 \times \tau_1 \times \tau_2)(\Gamma_R) \\
			%		&=\{(\tau_0(p_1),\tau_1(p_2),\tau_2(\sigma_M(p_1,p_2)) \mid (p_1,p_2) \in E\} \\
			%		&=\{(q_1,q_2,\sigma'_M(q_1,q_2)) \mid (q_1,q_2) \in E'\}.
			%	\end{align*}
		Thus $A'$ satisfies {\rm (G1)} with $\cP(A')=(E',\sigma')$
		where $E'=\pi'_{12}(\cV(R'))$ and $\sigma'=\pi'_{23}{\pi'_{12}}^{-1}$.
		
		(2) Assume that $A$ satisfies {\rm (G2)} with $A=\cA(E,\sigma)$.
		Using the same notation as the proof of (1), we define
		\begin{center}
			$E':=(\tau_0 \times \tau_{1})(E)$ and $\sigma':=(\tau_{1} \times \tau_2) \circ \sigma \circ (\tau_0 \times \tau_1)^{-1}$.
		\end{center}
		For any $(p_1,p_2) \in E$, it holds that
		\begin{align*}
			(\pi_1 \circ \sigma')(\tau_0(p_1),\tau_1(p_2))&=(\pi_1 \circ (\tau_1 \times \tau_2) \circ \sigma)(p_1,p_2) \\
			&=(\pi_1 \circ (\tau_1 \times \tau_2))(p_2,(\pi_2 \circ \sigma)(p_1,p_2)) \\
			&=\tau_1(p_2)=\pi_2(\tau_0(p_1),\tau_1(p_2)),
		\end{align*}
		so $(E',\sigma')$ is a geometric pair.
		%We take the same geometric pair $(E',\sigma')$ which is given in the proof of (1). We show that $A'=\cA(E',\sigma')$.
		Take $g \in V^{\otimes 3}$. It follows from Lemma \ref{lem.zero} that
		\begin{align*}
			&g(q_1,q_2,(\pi_2 \circ \sigma')(q_1,q_2))=0 \quad \forall (q_1,q_2) \in E' \\
			\Longleftrightarrow\,\,\, &g(\tau_0(p_1),\tau_1(p_2),(\pi_2\circ\sigma')(\tau_0(p_1),\tau_1(p_2)))=0 \quad \forall (p_1,p_2) \in E \\
			\Longleftrightarrow\,\,\, &g(\tau_0(p_1),\tau_1(p_2),(\pi_2\circ(\tau_1 \times \tau_2)\circ\sigma)(p_1,p_2))=0 \quad \forall (p_1,p_2) \in E \\
			\Longleftrightarrow\,\,\, &g(\tau_0(p_1),\tau_1(p_2),\tau_2((\pi_2\circ\sigma)(p_1,p_2)))=0 \quad \forall (p_1,p_2) \in E \\
			\Longleftrightarrow\,\,\, &((\phi_0|_V \otimes \phi_1|_V \otimes \phi_2|_V)(g))(p_1,p_2,(\pi_2 \circ \sigma)(p_1,p_2))=0 \quad \forall (p_1,p_2) \in E \\
			\Longleftrightarrow\,\,\, &(\phi_0|_V \otimes \phi_1|_V \otimes \phi_2|_V)(g) \in R \\
			\Longleftrightarrow\,\,\, &g \in R'.
		\end{align*}
		Thus $A'$ satisfies {\rm (G2)} with $A'=\cA(E',\sigma')$.
		
		(3) Assume that $A$ is geometric.
		By the proofs of (1) and (2), it holds that $A'$ satisfies {\rm (G1)} and {\rm (G2)} with $A'=\cA(\cP(A'))$, so
		$A'$ is geometric.
		%Let $f \in R'$. Take $(q_1,q_2) \in E'$ and write $(q_1,q_2)=(\tau_0(p_1),\tau_1(p_2))$ where $(p_1,p_2) \in E$.
		%Since
		%$(\phi_0|_V \otimes \phi_1|_V \otimes \phi_2|_V)(f) \in R$, we have that
		%\begin{center}
		%	$((\phi_0|_V \otimes \phi_1|_V \otimes \phi_2|_V)(f))(p_1,p_2,\sigma_M(p_1,p_2))=0$.
		%\end{center}
		%By Lemma \ref{lem.zero}, it holds that
		%\begin{center}
		%	$f(q_1,q_2,\sigma'_M(q_1,q_2))=0$.
		%\end{center}
		%Conversely, let $f \in V^{\otimes 3}$ such that
		%\begin{center}
		%	$f(q_1,q_2,\sigma'_M(q_1,q_2))=0$
		%\end{center}
		%for any $(q_1,q_2) \in E'$.
	\end{proof}

	The following results are useful to study geometric algebras. 
	
	\begin{theorem}[{cf. \cite[Lemma 2.5]{MU}}]\label{thm.G1}
		Let $A=T(V)/(R), A'=T(V)/(R')$ be cubic algebras
		satisfying {\rm (G1)} with $\cP(A)=(E,\sigma), \cP(A')=(E',\sigma')$.
		\begin{enumerate}
			\item
			If $A \cong A'$ as graded algebras, then there exists an automorphism $\tau$ of $\PP(V^{\ast})$
			such that $\tau \times \tau$ restricts to an isomorphism $\tau \times \tau: E \to E'$ and
			$$\xymatrix{
				E\ar[r]^-{\tau \times \tau}\ar[d]_-{\sigma}&E'\ar[d]^-{\sigma'}\\
				E\ar[r]_-{\tau \times \tau}&E'
			}$$
			commutes.
			
			\item
			If $\GrMod A \cong \GrMod A'$, then there exists a sequence of automorphisms $\tau_n$ of $\PP(V^{\ast})$
			such that $\tau_n \times \tau_{n+1}$ restricts to an isomorphism
			$\tau_n \times \tau_{n+1}: E \to E'$ and
			$$\xymatrix@C=40pt{
				E\ar@<0.8ex>[r]^-{\tau_n \times \tau_{n+1}}\ar[d]_-{\sigma}&E'\ar[d]^-{\sigma'}\\
				E\ar[r]_-{\tau_{n+1} \times \tau_{n+2}}&E'
			}$$
			commutes for any $n \in \ZZ$.
		\end{enumerate}
	\end{theorem}
	
	\begin{proof}
		%	Since (1) is the special case of (2), it is enough to show (2).
		%	(1) Assume that $A \cong A'$. Let $\varphi:A' \to A$ be a graded algebra isomorphism.
		%	We denote by $\varphi|_V$ the restircted map of $\varphi$ and define $\tau:=\overline{\varphi|_V^{\ast}}$.
		%	Since $\varphi$ is a graded algebra isomorphism,
		%	\begin{center}
			%		$(\varphi|_V \otimes \varphi|_V \otimes \varphi|_V)(R')=R$.
			%	\end{center}
		%	By Lemma \ref{lem.zeroR}, it holds that
		%	%\begin{center}
		%		$\Gamma_{R'}=(\tau \times \tau \times \tau)(\Gamma_R)$.
		%	%\end{center}
		%	Since $\cP(A)=(E,\sigma)$ and $\cP(A')=(E',\sigma')$, we have that $E'=(\tau \times \tau)(E)$ and
		%	\begin{center}
			%		$\sigma'_M \circ (\tau \times \tau)=\tau \circ \sigma_M$.
			%	\end{center}
		%	Thus the statement holds.
		We will give a proof for $(2)$. The proof of $(1)$ are similar.
		
		(2) As is explained in the proof of Theorem \ref{thm.GM} (1), if $\GrMod A \cong \GrMod A'$, then
		there exists a sequence of automorphisms $\tau_n$ of $\PP(V^{\ast})$
		such that the following diagrams
		%\begin{center}
		%	$\Gamma_{R'}=(\tau_n \times \tau_{n+1} \times \tau_{n+2})(\Gamma_R)$
		%\end{center}
		\begin{center}
			$\xymatrix@C=40pt{
				\cV(R)\ar[r]^-{\tau_n \times \tau_{n+1} \times \tau_{n+2}}_{\cong}\ar[d]_-{\pi_{12}}&\cV(R')\ar[d]^-{\pi'_{12}}\\
				\pi_{12}(\cV(R))\ar[r]^-{\cong}_-{\tau_n \times \tau_{n+1}}&\pi'_{12}(\cV(R')),
			}$
			$\xymatrix@C=40pt{
				\cV(R)\ar[r]^-{\tau_n \times \tau_{n+1} \times \tau_{n+2}}_-{\cong}\ar[d]_-{\pi_{23}}&\cV(R')\ar[d]^-{\pi'_{23}}\\
				\pi_{23}(\cV(R))\ar[r]^-{\cong}_-{\tau_{n+1} \times \tau_{n+2}}&\pi'_{23}(\cV(R'))
			}$
		\end{center}
		commute for any $n \in \ZZ$.
		Since $\cP(A)=(E,\sigma)$ and $\cP(A')=(E',\sigma')$,
		\begin{center}
			$\xymatrix@C=40pt{
				E \ar[r]^-{\tau_n \times \tau_{n+1}}_-{\cong} \ar[d]_-{\sigma} & E' \ar[d]^-{\sigma'} \\
				E \ar[r]^-{\cong}_-{\tau_{n+1} \times \tau_{n+2}} & E'
			}$
		\end{center}
		%we have that $E'=(\tau_n \times \tau_{n+1})(E)$ and
		%\begin{center}
		%	$\sigma'_M \circ (\tau_n \times \tau_{n+1})=\tau_{n+2} \circ \sigma_M$
		%\end{center}
		commutes for any $n \in \ZZ$.
		Thus the statement holds.
	\end{proof}
	
	\begin{theorem}[{cf. \cite[Lemma 2.6]{MU}}]\label{thm.G2}
		Let $A=T(V)/(R), A'=T(V)/(R')$ be cubic algebras
		satisfying {\rm (G2)} with $A=\cA(E,\sigma), A'=\cA(E',\sigma')$.
		\begin{enumerate}
			\item
			If there exists an automorphism $\tau$ of $\PP(V^{\ast})$
			such that $\tau \times \tau$ restricts to an isomorphism $\tau \times \tau: E \to E'$ and
			$$\xymatrix{
				E\ar[r]^-{\tau \times \tau}\ar[d]_-{\sigma}&E'\ar[d]^-{\sigma'}\\
				E\ar[r]_-{\tau \times \tau}&E'
			}$$
			commutes, then $A \cong A'$ as graded algebras.
			
			\item
			If there exists a sequence of automorphisms $\tau_n$ of $\PP(V^{\ast})$
			such that $\tau_n \times \tau_{n+1}$ restricts to an isomorphism
			$\tau_n \times \tau_{n+1}: E \to E'$ and
			$$\xymatrix@C=40pt{
				E\ar@<0.8ex>[r]^-{\tau_n \times \tau_{n+1}}\ar[d]_-{\sigma}&E'\ar[d]^-{\sigma'}\\
				E\ar[r]_-{\tau_{n+1} \times \tau_{n+2}}&E'
			}$$
			commutes for any $n \in \ZZ$, then $\GrMod A \cong \GrMod A'$.
		\end{enumerate}
	\end{theorem}
	
	\begin{proof}
		We will give a proof for $(2)$. The proof of $(1)$ are similar.
		%Since (1) is the special case of (2), it is enough to show (2).
		%(1) Let $\tau$ be an automorphism of $\PP(V^{\ast})$ such that
		%$\tau \times \tau$ restricts to an isomorphism $\tau \times \tau: E \to E'$ and
		%$$\xymatrix{
			%	E\ar[r]^-{\tau \times \tau}\ar[d]_-{\sigma}&E'\ar[d]^-{\sigma'}\\
			%	E\ar[r]_-{\tau \times \tau}&E'
			%}$$
		%commutes. Remark that from the second condition,
		%\begin{center}
		%	$\sigma'_M \circ (\tau \times \tau)=\tau \circ \sigma_M$.
		%\end{center}
		%Since $\tau$ is an automorphism of $\PP(V^{\ast})$, there exists $\phi \in \GL(V)$ such that $\tau=\overline{\phi^{\ast}}$.
		%Take $g \in V^{\otimes 3}$. By the condition {\rm (G2)} and Lemma \ref{lem.zero},
		%\begin{align*}
		%	& g \in R' \\
		%	\Longleftrightarrow\,\,\, &g(q_1,q_2,\sigma'_M(q_1,q_2))=0 \quad \forall (q_1,q_2) \in E' \\
		%	\Longleftrightarrow\,\,\, &g(\tau(p_1),\tau(p_2),\sigma'_M(\tau(p_1),\tau(p_2)))=0 \quad \forall (p_1,p_2) \in E \\
		%	\Longleftrightarrow\,\,\, &((\phi \otimes \phi \otimes \phi)(g))(p_1,p_2,\sigma_M(p_1,p_2))=0 \quad \forall (p_1,p_2) \in E \\
		%	\Longleftrightarrow\,\,\, &(\phi \otimes \phi \otimes \phi)(g) \in R,
		%\end{align*}
		%so we have that $(\phi \otimes \phi \otimes \phi)(R')=R$.
		%Therefore, $\phi$ induces a graded algebra isomorphism from $A'$ to $A$.
		
		(2) Let $\{\tau_n\}_{n \in \ZZ}$ be a sequence of automorphisms of $\PP(V^{\ast})$
		such that for any $n \in \ZZ$, $\tau_n \times \tau_{n+1}$ restricts to an isomorphism
		$\tau_n \times \tau_{n+1}: E \to E'$ and
		$$\xymatrix@C=40pt{
			E\ar@<0.8ex>[r]^-{\tau_n \times \tau_{n+1}}\ar[d]_-{\sigma}&E'\ar[d]^-{\sigma'}\\
			E\ar[r]_-{\tau_{n+1} \times \tau_{n+2}}&E'
		}$$
		commutes.
		%Remark that from the second condition, it holds that
		%\begin{center}
		%	$\sigma'_M \circ (\tau_n \times \tau_{n+1})=\tau_{n+2} \circ \sigma_M$
		%\end{center}
		%for any $n \in \ZZ$.
		For each $n \in \ZZ$, there exists $\phi_n \in \GL(V)$ such that $\tau_n=\overline{\phi_n^{\ast}}$.
		Take $g \in V^{\otimes 3}$. By the condition {\rm (G2)} and Lemma \ref{lem.zero}, it holds that
		\begin{align*}
			&g \in R' \\
			\Longleftrightarrow\,\,\, &g(q_1,q_2,(\pi_2 \circ \sigma')(q_1,q_2))=0 \quad \forall (q_1,q_2) \in E' \\
			\Longleftrightarrow\,\,\, &g(\tau_n(p_1),\tau_{n+1}(p_2),(\pi_2 \circ \sigma')(\tau_n(p_1),\tau_{n+1}(p_2)))=0 \quad \forall (p_1,p_2) \in E \\
			\Longleftrightarrow\,\,\, &g(\tau_n(p_1),\tau_{n+1}(p_2),(\pi_2 \circ (\tau_{n+1} \times \tau_{n+2}) \circ \sigma)(p_1,p_2))=0 \quad \forall (p_1,p_2) \in E \\
			\Longleftrightarrow\,\,\, &((\phi_n \otimes \phi_{n+1} \otimes \phi_{n+2})(g))(p_1,p_2,(\pi_2 \circ \sigma)(p_1,p_2))=0 \quad \forall (p_1,p_2) \in E \\
			\Longleftrightarrow\,\,\, &(\phi_n \otimes \phi_{n+1} \otimes \phi_{n+2})(g) \in R 
		\end{align*}
		for any $n \in \ZZ$, so we have that
		\begin{center}
			$(\phi_{n} \otimes \phi_{n+1} \otimes \phi_{n+2})(R')=R$
		\end{center}
		for any $n \in \ZZ$.
		Define graded linear isomorphisms $\Phi_n:T(V) \to T(V)$ by
		\begin{center}
			$\Phi_n|_{V^{\otimes l}}(v_1 \otimes v_2 \otimes \cdots \otimes v_l)=\phi_{n}(v_1) \otimes \phi_{n+1}(v_2) \otimes \cdots \otimes \phi_{n+l-1}(v_l)$
		\end{center}
		for $v_i \in V$. Since $(\phi_{n} \otimes \phi_{n+1} \otimes \phi_{n+2})(R')=R$, $\Phi_n$ induces well-defined graded linear isomorphisms
		$\Phi_n: A' \to A$ for any $n \in \ZZ$. By the definition of $\Phi_n$, it is clear that
		\begin{center}
			$\Phi_{n}(ab)=\Phi_{n}(a)\Phi_{n+m}(b)$
		\end{center}
		for any $l,m,n \in \ZZ$ and $a \in A'_m$, $b \in A'_l$.
		By Lemma \ref{lem.Zhang}, $A'$ is isomorphic to a twisted algebra of $A$.
		Hence, it follows from Theorem \ref{thm.Zhang} that $\GrMod A \cong \GrMod A'$.
	\end{proof}

	\begin{definition}\label{defin.2PE}
		Let $E$ and $E'$ be projective varieties in $\PP(V^{\ast}) \times \PP(V^{\ast})$
		where $V$ is a finite-dimensional vector space.
		\begin{enumerate}
			\item
			We say that $E$ and $E'$
			are {\it equivalent},
			denoted by $E \sim E'$, if $E'=(\tau_{1} \times \tau_2)(E)$ for some $\tau_{1},\tau_2 \in \Aut_k \PP(V^{\ast})$.
			
			\item
			We say that $E$ and $E'$ are {\it $2$-equivalent}, denoted by $E \sim_2 E'$, if $E'=(\tau \times \tau)(E)$
			for some $\tau \in \Aut_k \PP(V^{\ast})$.
		\end{enumerate}
	\end{definition}
	
	It is clear that if $E$ and $E'$ are $2$-equivalent, then they are equivalent.
	Let $A=T(V)/(R)$ and $A'=T(V)/(R')$ be geometric algebras with $\cP(A)=(E,\sigma)$ and $\cP(A')=(E',\sigma')$.
	If $A \cong A'$ (resp. $\GrMod A \cong \GrMod A'$), then $E$ and $E'$ are $2$-equivalent (resp. equivalent) by Theorem \ref{thm.G1}.
	As the first step of classification of geometric algebras up to graded algebra isomorphism (resp. graded Morita equivalence),
	we need to classify projective varieties up to $2$-equivalence (resp. equivalence).
	By the way, if $E \sim_2 E'$, that is, $E'=(\tau \times \tau)(E)$ for some $\tau \in \Aut_k \PP(V^{\ast})$, then
	we define an automorphism $\sigma''$ of $E'$ by
	\begin{center}
		$\sigma'':=(\tau \times \tau) \circ \sigma \circ (\tau \times \tau)^{-1}$.
	\end{center}
	The pair $(E',\sigma'')$ is a geometric pair, and
	the diagram
	$$\xymatrix{
		E\ar[r]^-{\tau \times \tau}\ar[d]_-{\sigma}&E'\ar[d]^-{\sigma''}\\
		E\ar[r]_-{\tau \times \tau}&E'
	}$$
	commutes, so $A \cong \cA(E',\sigma'')$ by Theorem \ref{thm.G2}.
	This means that we may assume that $E=E'$ to classify geometric algebras up to graded algebra isomorphism.
	%%%%%%%%%%%%%%%%%%%%%%%%%%%%%%%%%%%%%%%%%%%%%%%%%%%%%%%%%%%%%%%%%%%%%%%%%%%%%%%%%%%%%%%%%%%%%%%%%%%%%%%%%%%%%%%%%%%%%%%%%
	\section{Classification of $3$-dimensional cubic AS-regular algebras}
	%In \cite{ATV}, Artin-Tate-Van den Bergh found a nice one-to-one correspondence between the set of $3$-dimensional AS-regular algebras
	%finitely generated in degree $1$
	%and the set of pairs $(E,\sigma)$ where $E$ is , $\sigma$ is an automorphism of $E$.
	%When $A$ is a $3$-dimensional cubic AS-regular algebra,
	%In the cubic case, a pair $(E,\sigma)$ is a geometric pair defined in Section \ref{sec.GA} where
	%$E$ is $\PP^1 \times \PP^1$ or a divisor of bidegree $(2,2)$ in $\PP^1 \times \PP^1$.
	%$\cL$ is an invertible sheaf on $E$.
	In \cite{ATV}, Artin-Tate-Van den Bergh found a nice geometric characterization
	of $3$-dimensional AS-regular algebras finitely generated in degree $1$.
	In this paper, we focus on the cubic case.
	\begin{theorem}[{\cite{ATV}}]\label{thm.ATV2}
		Every $3$-dimensional cubic AS-regular algebra $A$ is a geometric algebra.
		Moreover, $E$ is $\PP^1 \times \PP^1$ or a divisor of bidegree $(2,2)$ in $\PP^1 \times \PP^1$.
	\end{theorem}
	
	%\begin{remark}
	%	The original theorem proved by Artin-Tate-Van den Bergh \cite{ATV} contains the case that $E$ is non-reduced, however,
	%	if the point variety $E$ of $A$ is non-reduced, then $A$ does not satisfy the condition (G2), that is, $A$ is not geometric.
	%\end{remark}
	
	In this paper, we study two cases when $E=\PP^1 \times \PP^1$ and $E$ is a union of irreducible divisors of bidegree $(1,1)$ in $\PP^1 \times \PP^1$. 
	For each case, we
	\begin{itemize}
		\item[{\rm (I)}] give a complete list of defining relations of $3$-dimensional cubic AS-regular algebras,
		\item[{\rm (II)}] classify them up to graded algebra isomorphism in terms of their defining relations, and
		\item[{\rm (III)}] classify them up to graded Morita equivalence in terms of their defining relations.
	\end{itemize}
	To give a complete list of defining relations of $3$-dimensional cubic AS-regular algebras $\cA(E,\sigma)$ for each case,
	we need to find all geometric pairs $(E,\sigma)$ corresponding to each variety $E$.
	%Let $E \subset \PP^1 \times \PP^1$ be a projective variety and $\pi_i:\PP^1 \times \PP^1 \to \PP^1$ $i$-th projections where $i=1,2$.
	%We set the following notation:
	%\begin{center}
	%	$\Aut_k^R E:=\{\sigma \in \Aut_k E \mid \pi_1 \circ \sigma=\pi_2 \}$.
	%\end{center}
	By the definition of a geometric pair, to determine all geometric pairs $(E,\sigma)$ is reduced to find $\Aut_k^G E$.
	Remark that since the identity $\id_E$ does not belong to $\Aut_k^G E$, $\Aut_k^G E$ is not a subgroup of $\Aut_k E$.
	Since $\Aut_k \PP^1 \cong \PGL_2(k)$, we often identify $\tau \in \Aut_k \PP^1$ with the representing matrix $\tau \in \PGL_2(k)$.
	
	We first treat the case $E=\PP^1 \times \PP^1$.
	We denote by $\nu$ an automorphism of $\PP^1 \times \PP^1$ defined by
	%\begin{center}
	$\nu(p,q)=(q,p)$ for $(p,q) \in \PP^1 \times \PP^1$.
	%\end{center}
	\begin{lemma}[{\cite[p.99]{Ma}}]\label{lem.AutPP}
		%Let $\nu$ be an automorphism of $\PP^1 \times \PP^1$ defined by
		%\begin{center}
		%	$\nu(p,q)=(q,p) \quad ((p,q) \in \PP^1 \times \PP^1)$.
		%\end{center}
		Every automorphism of $\PP^1 \times \PP^1$ is written as $\tau_{1} \times \tau_2$ or $(\tau_{1} \times \tau_2) \circ \nu$
		for some $\tau_{1}, \tau_2 \in \Aut_k \PP^1$. Moreover, it holds that
		\begin{center}
			$\Aut_k (\PP^1 \times \PP^1) =(\Aut_k \PP^1 \times \Aut_k \PP^1) \rtimes \langle \nu \rangle$.
		\end{center}
	\end{lemma}
	
	%It is easy to show the following result.
	
	%\begin{lemma}\label{lem.AutPP2}
	%	Let $\nu$ be an automorphism of $\PP^1 \times \PP^1$ defined by
	%	\begin{center}
		%		$\nu(p,q)=(q,p) \quad ((p,q) \in \PP^1 \times \PP^1)$.
		%	\end{center}
	%	Then $\Aut_k (\PP^1 \times \PP^1) =(\Aut_k \PP^1 \times \Aut_k \PP^1) \rtimes \langle \nu \rangle$.
	%\end{lemma}
	
	\begin{lemma}\label{lem.AutPP3}
		%Let $\nu$ be an automorphism of $\PP^1 \times \PP^1$ defined by
		%\begin{center}
		%	$\nu(p,q)=(q,p) \quad ((p,q) \in \PP^1 \times \PP^1)$.
		%\end{center}
		$\Aut_k^{G} (\PP^1 \times \PP^1)=\{(\id_{\PP^1} \times \tau) \circ \nu \mid \tau \in \Aut_k \PP^1\}$.
	\end{lemma}
	
	\begin{proof}
		%Let $\tau \in \Aut_k \PP^1$ and $(p,q) \in \PP^1 \times \PP^1$.
		%Since 
		%\begin{center}
		%	$(\pi_1 \circ (\id_{\PP^1} \times \tau) \circ \nu)(p,q)=q=\pi_2(p,q)$,
		%\end{center}
		%we have that
		%$\{(\id_{\PP^1} \times \tau) \circ \nu \mid \tau \in \Aut_k \PP^1\} \subset \Aut_k^R (\PP^1 \times \PP^1)$.
		Let $\sigma \in \Aut_k^G (\PP^1 \times \PP^1)$. By Lemma \ref{lem.AutPP}, $\sigma$ is written as
		%\begin{center}
		$\sigma=(\tau_1 \times \tau_2) \circ \nu^i$
		%\end{center}
		where $\tau_1, \tau_2 \in \Aut_k \PP^1$ and $i=0,1$.
		If $i=0$, then since $\pi_1 \circ (\tau_1 \times \tau_2)=\pi_2$, it holds that $\tau_{1}(p)=q$ for any $(p,q) \in \PP^1 \times \PP^1$,
		but it contradicts that $\tau_1 \in \Aut_k \PP^1$.
		Assume that $i=1$.
		Since $(\pi_1 \circ (\tau_1 \times \tau_2) \circ \nu)(p,q)=\tau_1(q)$, we have that
		$\pi_1 \circ (\tau_1 \times \tau_2) \circ \nu=\pi_2$ if and only if $\tau_{1}=\id_{\PP^1}$.
		%it holds that $\tau_{1}(q)=q$ for any $(p,q) \in \PP^1 \times \PP^1$,
		%so $\tau_{1}=\id_{\PP^1}$. Thus, we have that
		%\begin{center}
		%	$\Aut_k^{R} (\PP^1 \times \PP^1)=\{(\id_{\PP^1} \times \tau) \circ \nu \mid \tau \in \Aut_k \PP^1\}$.
		%\end{center}
		Hence, the statement holds.
	\end{proof}
	
	We next treat the case when $E$ is a union of irreducible divisors of bidegree $(1,1)$ in $\PP^1 \times \PP^1$.
	
	\begin{lemma}\label{lem.CC}
		Let $C=\cV(f) \subset \PP^1 \times \PP^1$ where $f \in k[x_1,y_1] \circ k[x_2,y_2]$.
		Assume that $C$ is irreducible.
		Then $\deg C=(1,1)$ if and only if there exists $\tau \in \Aut_k \PP^1$ such that
		%\begin{center}
		$C=C_{\tau}:=\{(p,\tau(p)) \mid p \in \PP^1\}$.
		%\end{center}
	\end{lemma}
	
	\begin{proof}
		For $g=ax_1y_2+by_1x_2+cx_1x_2+dy_1y_2 \in k[x_1,y_1] \circ k[x_2,y_2]$ with $a,b,c,d \in k$, it follows from calculation that 
		$\cV(g)$ is irreducible if and only if $ab-cd \neq 0$.
		
		We assume that $\deg C=(1,1)$. Since $C=\cV(f)$ is irreducible, we have that $ab-cd \neq 0$, so
		$\tau:=
		\begin{pmatrix}
			-a & -d \\
			c & b
		\end{pmatrix} \in \Aut_k \PP^1$.
		Therefore, we have that $C=C_{\tau}$.
		Conversely, we assume that there exists $\tau=
		\begin{pmatrix}
			\a & \b \\
			\g & \d
		\end{pmatrix}
		\in \Aut_k \PP^1$ such that $C=C_{\tau}$.
		Then we set $g:=-\a x_1y_2+\d y_1x_2+\g x_1x_2-\b y_1y_2$.
		Since $\a\d-\b\c \neq 0$, $\cV(g)$ is irreducible and $C=\cV(g)$, so $\deg C=(1,1)$.
		Hence, the statement holds. 
		\end{proof}
		
		By Lemma \ref{lem.CC}, if $E$ is a union of irreducible divisors of bidegree $(1,1)$ in $\PP^1 \times \PP^1$, then
		$E$ is written as $E=C_{\tau_1} \cup C_{\tau_2}$ for some $\tau_1,\tau_2 \in \Aut_k \PP^1$.
		We remark that $C_{\tau_{1}} \cap C_{\tau_2} \neq \emptyset$
		because $k$ is an algebraically closed field. For $P,Q \in \PGL_2(k)$,
		we denote by $P \sim Q$ when two matrices are similar in $\PGL_2(k)$.
		\begin{theorem}\label{thm.Bez}
			Let $E=C_{\tau_1} \cup C_{\tau_2}$ be a union of two irreducible divisors of bidegree $(1,1)$ in $\PP^1 \times \PP^1$
			where $\tau_1, \tau_2 \in \Aut_k \PP^1$.
			%{\tcr Assume that $C_{\tau_{1}} \cap C_{\tau_2} \neq \emptyset$.}
			Then one of the following statements holds:
			\begin{enumerate}
				\item[{\rm (1)}] $|C_{\tau_1} \cap C_{\tau_2}|=2$ if and only if $\tau_2^{-1}\tau_1 \sim
				\begin{pmatrix}
					\l & 0 \\
					0 & 1 
				\end{pmatrix}$ for some $\l \in k$ with $\l \neq 0,1$.
				\item[{\rm (2)}] $|C_{\tau_1} \cap C_{\tau_2}|=1$ if and only if $\tau_2^{-1}\tau_1 \sim
				\begin{pmatrix}
					1 & 1 \\
					0 & 1
				\end{pmatrix}$.
				\item[{\rm (3)}] $|C_{\tau_1} \cap C_{\tau_2}|=\infty$ if and only if
				$\tau_1=\tau_2$.
				%$\tau_2^{-1}\tau_1 \sim
				%\begin{pmatrix}
				%	1 & 0 \\
				%	0 & 1
				%\end{pmatrix}$.
			\end{enumerate}
		\end{theorem}
		
		\begin{proof}
			For $(p,q) \in \PP^1 \times \PP^1$, $(p,q) \in C_{\tau_1} \cap C_{\tau_2}$ if and only if $q=\tau_{1}(p)$ and $\tau_2^{-1}\tau_1(p)=p$
			if and only if $q=\tau_1(p)$ and a representative of $p$ is an eigenvector of a representative of $\tau_2^{-1}\tau_1$.
			Viewing $\tau_2^{-1}\tau_1$ as an element of $\PGL_2(k)$,
			it is similar to one of the following matrices;
			\begin{center}
				(1) $\begin{pmatrix}
					\l & 0 \\
					0 & 1
				\end{pmatrix}$ \ (for some $0,1 \neq \l \in k$),\ (2) $\begin{pmatrix}
				1 & 1 \\
				0 & 1
				\end{pmatrix}$,\ (3) $\begin{pmatrix}
				1 & 0 \\
				0 & 1
				\end{pmatrix}$.
			\end{center}
			
			(1)  
			$\tau_2^{-1}\tau_1 \sim
			\begin{pmatrix}
				\l & 0 \\
				0 & 1
			\end{pmatrix}$ in $\PGL_2(k)$ for some $\l \in k$ with $\l \neq 0,1$
			if and only if a representative of $\tau_2^{-1}\tau_1$ has two distinct non-zero eigenvectors, that is, $|C_{\tau_1} \cap C_{\tau_2}|=2$.
			
			(2) 
			$\tau_2^{-1}\tau_1 \sim
			\begin{pmatrix}
				1 & 1 \\
				0 & 1
			\end{pmatrix}$ in $\PGL_2(k)$
			if and only if a representative of $\tau_2^{-1}\tau_1$ has only one non-zero eigenvector, that is, $|C_{\tau_1} \cap C_{\tau_2}|=1$.
			
			(3)
			$\tau_2^{-1}\tau_1 \sim
			\begin{pmatrix}
				1 & 0 \\
				0 & 1
			\end{pmatrix}$ in $\PGL_2(k)$ if and only if $\tau_1=\tau_2$, that is, $|C_{\tau_1} \cap C_{\tau_2}|=\infty$.
		\end{proof}
		
		%Let $E=C_{\tau_1} \cup C_{\tau_2}$ be a union of two irreducible divisors of bidegree $(1,1)$ in $\PP^1 \times \PP^1$
		%where $\tau_1, \tau_2 \in \Aut_k \PP^1$. By Theorem \ref{thm.Bez}, $E$ is reduced if and only if $\tau_1 \neq \tau_2$.
		%For the rest paper,
		%when we treat the case that $E=C_{\tau_1} \cup C_{\tau_2}$ is a union of two irreducible divisors of bidegree $(1,1)$ in $\PP^1 \times \PP^1$,
		%we tacitly assume that $\tau_1 \neq \tau_2$.
		
		In this paper, we define the types of geometric pairs $(E,\sigma)$ as follows:
		\begin{description}
			\item[{\rm (1) Type P}] $E$ is $\PP^1 \times \PP^1$, and $\sigma=(\id_{\PP^1} \times \tau) \circ \nu \in \Aut_k^G(\PP^1 \times \PP^1)$
			(Type P is divided into Type P$_i$ ($i=1,2$) in terms of the Jordan canonical form of $\tau$ ).
			\item[{\rm (2) Type S}] $E=C_{\tau_1} \cup C_{\tau_2}$ is a union of two irreducible divisors of bidegree $(1,1)$ in $\PP^1 \times \PP^1$
			such that $|C_{\tau_1} \cap C_{\tau_2}|=2$. Type S is divided into Type S$_i$ ($i=1,2$) in terms of $\sigma$;
			\item[{\rm (2-1) Type S$_1$}] $\sigma$ fixes each component.
			\item[{\rm (2-2) Type S$_2$}] $\sigma$ switches each component.
			\item[{\rm (3) Type T}] $E=C_{\tau_1} \cup C_{\tau_2}$ is a union of two irreducible divisors of bidegree $(1,1)$ in $\PP^1 \times \PP^1$
			such that $|C_{\tau_1} \cap C_{\tau_2}|=1$
			Type T is divided into Type T$_i$ ($i=1,2$) in terms of $\sigma$;
			\item[{\rm (3-1) Type T$_1$}] $\sigma$ fixes each component.
			\item[{\rm (3-2) Type T$_2$}] $\sigma$ switches each component.
			%\item[{\rm (4) Type WL}] $E=C_{\tau_1} \cup C_{\tau_2}$ is a union of two irreducible divisors of bidegree $(1,1)$ in $\PP^1 \times \PP^1$
			%such that $|C_{\tau_1} \cap C_{\tau_2}|=\infty$.
		\end{description}
		
		%\begin{remark}
		%	Let $E=C_{\tau_1} \cup C_{\tau_2}$ is a union of two irreducible divisors of bidegree $(1,1)$ in $\PP^1 \times \PP^1$ with $\tau_1 \neq \tau_2$.
		%	If $\sigma \in \Aut_k^G E$, then the form of $\sigma$ is divided into two cases.
		%	The one is an automorphism $\sigma_1$ of $E$ which fixes each component.
		%	The other is an automorphism $\sigma_2$ of $E$ which switches each component.
		%	Therefore, we have that $\Aut_k^G E=\{\sigma_1,\sigma_2\}$.  
		%\end{remark}
		
		We introduce a useful notion to classify $E$ up to $2$-equivalence.
		
		\begin{definition}\label{def.sim}
			Let $\tau_1, \tau_2, \tau_1', \tau_2' \in \Aut_k \PP^1$.
			We say that two pairs $(\tau_1,\tau_2)$ and $(\tau'_1,\tau'_2)$ are \textit{similar}, denoted by $(\tau_1,\tau_2) \sim (\tau'_1,\tau'_2)$,
			if there exists $\mu \in \Aut_k \PP^1$ such that $\tau'_i=\mu^{-1}\tau_i\mu$ where $i=1,2$.
		\end{definition}
		
		\begin{lemma}\label{lem.Ctau}
			%Let $C_{\tau}$ and $C_{\tau'}$ be two irreducible divisors of bidegree $(1,1)$ in $\PP^1 \times \PP^1$.
			Let $\tau_1, \tau_2, \tau_1', \tau_2' \in \Aut_k \PP^1$.
			Then the following statements hold;
			\begin{enumerate}
				\item[{\rm (1)}] $(\tau'_1,\tau'_2) \sim (\tau_1,\tau_2)$ if and only if there exists $\mu \in \Aut_k \PP^1$ such that
				%\begin{center}
				$(\mu \times \mu)(C_{\tau'_i})=C_{\tau_i}$ where $i=1,2$.
				%\end{center}
				\item[{\rm (2)}] ${\tau'_2}^{-1}\tau'_1 \sim \tau_2^{-1}\tau_1$ if and only if there exist $\mu_1, \mu_2 \in \Aut_k \PP^1$ such that
				$(\mu_1 \times \mu_2)(C_{\tau'_i})=C_{\tau_i}$ where $i=1,2$.
			\end{enumerate}
		\end{lemma}
		
		\begin{proof}
			(1) Since $C_{\tau}$ is the graph of $\tau$, it follows that
			$(\tau'_1,\tau'_2) \sim (\tau_1,\tau_2)$ if and only if there exists $\mu \in \Aut_k \PP^1$ such that the diagram
			\begin{center}
				$\xymatrix{
					\PP^1 \ar[r]^-{\mu} \ar[d]_-{\tau'_i} & \PP^1 \ar[d]^-{\tau_i} \\
					\PP^1 \ar[r]_-{\mu} & \PP^1
				}$
			\end{center}
			commutes where $i=1,2$ if and only if there exists $\mu \in \Aut_k \PP^1$ such that $(\mu \times \mu)(C_{\tau'_i})=C_{\tau_i}$ where $i=1,2$.
			
			(2) Since $C_{\tau}$ is the graph of $\tau$, it follows that
			${\tau'_2}^{-1}\tau'_1 \sim \tau_2^{-1}\tau_1$ if and only if there exist $\mu_1,\mu_2 \in \Aut_k \PP^1$ such that the diagram
			\begin{center}
				$\xymatrix{
					\PP^1 \ar[r]^-{\mu_1} \ar[d]_-{\tau'_i} & \PP^1 \ar[d]^-{\tau_i} \\
					\PP^1 \ar[r]_-{\mu_2} & \PP^1
				}$
			\end{center}
			commutes where $i=1,2$ if and only if there exist $\mu_1,\mu_2 \in \Aut_k \PP^1$ such that $(\mu_1 \times \mu_2)(C_{\tau'_i})=C_{\tau_i}$ where $i=1,2$.
		\end{proof}
		
		%\begin{lemma}\label{lem.irred}
		%	Let $E \subset \PP^1 \times \PP^1$ be a projective variety and $\tau_1, \tau_2 \in \Aut_k \PP^1$.
		%	If $E$ is irreducible, then so is $(\tau_1 \times \tau_2)(E)$.
		%\end{lemma}
		
		%\begin{lemma}\label{lem.irred2}
		%	Let $C_{\tau}$ and $C_{\tau'}$ be two irreducible divisors of bidegree $(1,1)$ in $\PP^1 \times \PP^1$.
		%	Then the followings hold;
		%	\begin{enumerate}
			%		\item[{\rm (1)}] $C_{\tau} \sim_2 C_{\tau'}$ if and only if $\tau \sim \tau'$.
			%		\item[{\rm (2)}] $C_{\tau} \sim C_{\tau'}$ if and only if there exist $\mu_1, \mu_2 \in \Aut_k \PP^1$ such that
			%		\begin{center}
				%			$\mu_1\tau=\tau'\mu_2$.
				%		\end{center}
			%	\end{enumerate}
		%\end{lemma}
		
		%\begin{proof}
		%	It is sufficient to show (2).
		%	It follows that
		%	\begin{align*}
			%		&C_{\tau} \sim C_{\tau'} \\
			%		& \Longleftrightarrow\,\,\, (\mu_1 \times \mu_2)(C_{\tau})=C_{\tau'} \quad\exists \mu_1, \mu_2 \in \Aut_k \PP^1 \\
			%		& \Longleftrightarrow\,\,\, \mu_2\tau=\tau'\mu_1 \quad\exists \mu_1, \mu_2 \in \Aut_k \PP^1,
			%		%& \Longleftrightarrow\,\,\, \tau \sim \tau',
			%	\end{align*}
		%	so the statement holds.
		%\end{proof}
		
		\begin{lemma}\label{lem.equiv}
			Let $E=C_{\tau_1} \cup C_{\tau_2}$ and $E'=C_{\tau'_1} \cup C_{\tau'_2}$
			be unions of two irreducible divisors of bidegree $(1,1)$ in $\PP^1 \times \PP^1$
			where $\tau_i, \tau'_i \in \Aut_k \PP^1$ and $i=1,2$.
			Then the following statements hold:
			\begin{enumerate}
				\item[{\rm (1)}] $E \sim_2 E'$ if and only if $(\tau'_1,\tau'_2) \sim (\tau_1,\tau_2)$ or $(\tau'_1,\tau'_2) \sim (\tau_2,\tau_1)$.
				\item[{\rm (2)}] $E \sim E'$ if and only if $\tau_2'^{-1}\tau_1' \sim (\tau_2^{-1}\tau_1)^{\pm 1}$.
			\end{enumerate}
		\end{lemma}
		
		\begin{proof}
			Since $C_{\tau_1}, C_{\tau_2}$ are irreducible, it holds that
			$E \sim_2 E'$ if and only if there exists $\mu \in \Aut_k \PP^1$ such that
			\begin{align*}
				&(\mu \times \mu)(C_{\tau_1})=C_{\tau'_1} \text{ and } (\mu \times \mu)(C_{\tau_2})=C_{\tau'_2}, \text{ or } \\
				&(\mu \times \mu)(C_{\tau_1})=C_{\tau'_2} \text{ and } (\mu \times \mu)(C_{\tau_2})=C_{\tau'_1}.
			\end{align*}
			Similarly, $E \sim E'$ if and only if there exist $\mu_1,\mu_2 \in \Aut_k \PP^1$ such that
			\begin{align*}
				&(\mu_1 \times \mu_2)(C_{\tau_1})=C_{\tau'_1} \text{ and } (\mu_1 \times \mu_2)(C_{\tau_2})=C_{\tau'_2}, \text{ or } \\
				&(\mu_1 \times \mu_2)(C_{\tau_1})=C_{\tau'_2} \text{ and } (\mu_1 \times \mu_2)(C_{\tau_2})=C_{\tau'_1}.
			\end{align*}
			By Lemma \ref{lem.Ctau}, each statement holds.
		\end{proof}

		%The following proposition lists all possible defining relations of algebras $A$
		%and corresponding elements $w$ of $V^{\otimes 4}$ with $A=\cD(w)$
		%in each type.
		%up to graded algebra isomorphism.
		
		The following theorem lists all possible defining relations of algebras in each type up to
		isomorphism of graded algebra.
		%%%%%%%%%%%%%%%%%%%%%%%%%%%%%%%%%%%%%%%%%%%%%%%%%%%%%%%%%%%%%%%%%%%%%%%%%%%%%%%%%%%%%%%%%%%%%%%%%%%%
		\begin{theorem}\label{thm.main1}
			Let $A=\cA(E,\sigma)$ be a $3$-dimensional cubic AS-regular algebra. For each type the following table describes
			\begin{description}
				\item[{\rm(I)}] the defining relations of $A$, and
				\item[{\rm (II)}] the conditions to be graded algebra isomorphic in terms of their defining relations.
			\end{description}
			Moreover, every algebra listed in the following table is AS-regular.
			In the following table, if $X \neq Y$ or $i \neq j$, then Type $X_i$ algebra is not isomorphic to any Type $Y_j$ algebra.
		\end{theorem}
		
		\noindent{\renewcommand\arraystretch{1.5} 
			\begin{tabular}{|p{0.7cm}|p{5.0cm}|p{5.6cm}|}
				\multicolumn{3}{c}{Table 1 : Defining relations and conditon (II)}
				\\
				\hline
				{\footnotesize Type}  &\quad (I) defining relations
				&\quad (II) condition to be  \\ %\cline{2-2}
				&\quad ($\a,\b \in k$) %\\ \cline{2-2}
				&\quad graded algebra isomorphic \\ %\cline{2-2}
				\hline
		\end{tabular}}
		
		\noindent{\renewcommand\arraystretch{1.5} 
			\begin{tabular}{|p{0.7cm}|p{5.0cm}|p{5.6cm}|} \hline
				%%%%%%%%%%%%%%%%%%%%%%%%%%%%%%%%%%%%%%%%%%%%%%%%%%%%%%P_1
				P$_1$   &%\quad
				$
				\begin{cases}
					x^2y-\a yx^2, \\
					xy^2-\a y^2x \ \ \text{($\a \neq 0$)}
				\end{cases}
				$
				&\quad
				$\a'=\a^{\pm 1}$
				%\ \hfill --------------------- \hfill \rule{0pt}{10pt}
				\\ \hline
				%%%%%%%%%%%%%%%%%%%%%%%%%%%%%%%%%%%%%%%%%%%%%%%%%%%%%P_2
				P$_2$ &%\quad 
				$
				\begin{cases}
					x^2y-yx^2+yxy, \\
					xy^2-y^2x+y^3
				\end{cases}
				$
				&
				\ \hfill --------------------- \hfill \rule{0pt}{10pt}
				%\quad 
				%$\alpha'\beta'\gamma'=(\alpha\beta\gamma)^{\pm 1}$
				\\ \hline
				%%%%%%%%%%%%%%%%%%%%%%%%%%%%%%%%%%%%%%%%%%%%%%%%%%%%%%%%S_1
				S$_1$  &%\quad
				$
				\begin{cases}
					\a\b x^2y+(\a+\b) xyx+yx^2, \\
					\a\b xy^2+(\a+\b) yxy+y^2x
				\end{cases}
				$
				
				\text{($\a\b \neq 0$, $\a^2 \neq \b^2$)}
				&\quad
				$\{\a',\b'\}=\{\a,\b\}, \{\a^{-1},\b^{-1}\}$
				\\ \hline
				%%%%%%%%%%%%%%%%%%%%%%%%%%%%%%%%%%%%%%%%%%%%%%%%%%%%%%%%S_2
				S$_2$ &%\quad
				\begin{minipage}{150pt}
					$
					\begin{cases}
						xy^2+y^2x+(\a+\b)x^3, \\
						x^2y+yx^2+(\a^{-1}+\b^{-1})y^3
					\end{cases}
					$
					
					\text{($\a\b \neq 0$, $\a^2 \neq \b^2$)}
				\end{minipage}
				&\quad
				$\dfrac{\a'}{\b'}=\left(\dfrac{\a}{\b}\right)^{\pm 1}$
				%\ \hfill --------------------- \hfill \rule{0pt}{10pt}
				\\ \hline
			\end{tabular}}
		
	\noindent{\renewcommand\arraystretch{1.5} 
\begin{tabular}{|p{0.7cm}|p{5.0cm}|p{5.6cm}|} \hline
				%%%%%%%%%%%%%%%%%%%%%%%%%%%%%%%%%%%%%%%%%%%%%%%%%%%%%%%%T_1
				T$_1$ &%\quad
				\begin{minipage}{150pt}
					$
					\begin{cases}
						x^2y-2xyx+yx^2 \\
						\hfill -2(2\b-1)yxy \\
						\hfill +2(2\b-1)xy^2 \\
						\hfill +2\b(\b-1)y^3, \\
						xy^2-2yxy+y^2x 
					\end{cases}
					$
				\end{minipage}
				&\quad
				$\b'=\b, 1-\b$
				%\ \hfill --------------------- \hfill \rule{0pt}{10pt}
				\\ \hline
				%%%%%%%%%%%%%%%%%%%%%%%%%%%%%%%%%%%%%%%%%%%%%%%%%%%%%%%%T_2
				T$_2$ &%\quad
				\begin{minipage}{150pt}
					$
					\begin{cases}
						x^2y+2xyx+yx^2+2y^3, \\
						xy^2+2yxy+y^2x
					\end{cases}
					$ 
				\end{minipage}
				&
				\ \hfill --------------------- \hfill \rule{0pt}{10pt}
				\\ \hline
				%%%%%%%%%%%%%%%%%%%%%%%%%%%%%%%%%%%%%%%%%%%%%%%%%%%%%%%%WL_1
				%		WL$_1$ &%\quad
				%		\begin{minipage}{150pt}
					%			$
					%			\begin{cases}
						%				\a xy^2+y^2x-2\a yxy, \\
						%				\a^2x^2y+yx^2-2\a xyx 
						%			\end{cases}
					%			$
					%			
					%			\text{($\a \neq 0$)}
					%		\end{minipage}
				%		&\quad
				%		$\a'=\a^{\pm1}$
				%		%\ \hfill --------------------- \hfill \rule{0pt}{10pt}
				%		\\ \hline
				%%%%%%%%%%%%%%%%%%%%%%%%%%%%%%%%%%%%%%%%%%%%%%%%%%%%%%%%WL_2
				%		WL$_2$ &%\quad
				%		\begin{minipage}{150pt}
					%			$
					%			\begin{cases}
						%				xy^2+y^2x-2yxy, \\
						%				x^2y+yx^2+4xy^2 \\
						%				\hfill -2xyx-4yxy+2y^3
						%			\end{cases}
					%			$ 
					%		\end{minipage}
				%		&
				%		\ \hfill --------------------- \hfill \rule{0pt}{10pt}
				%		\\ \hline
		\end{tabular}}
		\vspace*{1em}
		%%%%%%%%%%%%%%%%%%%%%%%%%%%%%%%%%%%%%%%%%%%%%%%%%%%%%%%%%%%%%%%%%%%%%%%%%%%%%%%%%%%%
		
		The following theorem lists all possible defining relations of algebras in each type up to graded Morita equivalence.
		
		\begin{theorem}\label{thm.main2}
			Let $A=\cA(E,\sigma)$ be a $3$-dimensional cubic AS-regular algebra. For each type the following table describes
			\begin{description}
				\item[{\rm (I)}] the defining relations of $A$, and
				\item[{\rm (III)}] the conditions to be graded Morita equivalent in terms of their defining relations.
			\end{description}
			Moreover, every algebra listed in the following table is AS-regular.
			In the following table, if $X \neq Y$, then Type $X$ algebra is not graded Morita equivalent to any Type $Y$ algebra.
		\end{theorem}
		
		%%%%%%%%%%%%%%%%%%%%%%%%%%%%%%%%%%%%%%%%%%%%%%%%%%%%%%%%%%%%
		\noindent{\renewcommand\arraystretch{1.5} 
			\begin{tabular}{|p{0.7cm}|p{5.0cm}|p{5.6cm}|}
				\multicolumn{3}{c}{Table 2 : Defining relations and condition (III)}
				\\
				\hline
				{\footnotesize Type}  &\quad (I) defining relations
				&\quad (III) condition to be  \\ %\cline{2-2}
				&\quad ($\a, \b \in k$) %\\ \cline{2-2}
				&\quad graded Morita equivalent \\ %\cline{2-2}
				\hline
		\end{tabular}}
		
		\noindent{\renewcommand\arraystretch{1.5} 
			\begin{tabular}{|p{0.7cm}|p{5.0cm}|p{5.6cm}|} \hline
				%%%%%%%%%%%%%%%%%%%%%%%%%%%%%%%%%%%%%%%%%%%%%%%%%%%%%%P_1
				P   &%\quad
				$
				\begin{cases}
					x^2y-yx^2, \\
					xy^2-y^2x
				\end{cases}
				$
				&
				\ \hfill --------------------- \hfill \rule{0pt}{10pt}
				\\ \hline
				%%%%%%%%%%%%%%%%%%%%%%%%%%%%%%%%%%%%%%%%%%%%%%%%%%%%%S_1
				S &%\quad 
				$
				\begin{cases}
					\a\b x^2y+(\a+\b) xyx+yx^2, \\
					\a\b xy^2+(\a+\b) yxy+y^2x
				\end{cases}
				$
				
				\text{($\a\b \neq 0$, $\a^2 \neq \b^2$)}
				&\quad 
				$\dfrac{\a'}{\b'}=\left(\dfrac{\a}{\b}\right)^{\pm 1}$
				\\ \hline
				%%%%%%%%%%%%%%%%%%%%%%%%%%%%%%%%%%%%%%%%%%%%%%%%%%%%%%%%S_2
				%		S$_2$  &%\quad
				%		\begin{minipage}{150pt}
					%			$
					%			\begin{cases}
						%				xy^2+y^2x+(\a+1)x^3, \\
						%				x^2y+yx^2+(1+\a^{-1})y^3
						%			\end{cases}
					%			$
					%			
					%			\text{($\a^2 \neq 0,1$)}
					%		\end{minipage}
				%		&\quad
				%		$\a'=\a^{\pm 1}$
				%		%\ \hfill --------------------- \hfill \rule{0pt}{10pt}
				%		\\ \hline
				%%%%%%%%%%%%%%%%%%%%%%%%%%%%%%%%%%%%%%%%%%%%%%%%%%%%%%%%T
				T &%\quad
				\begin{minipage}{150pt}
					$
					\begin{cases}
						x^2y+yx^2+2xy^2 \\
						\hfill -2xyx-2yxy, \\
						xy^2+y^2x-2yxy
					\end{cases}
					$ 
				\end{minipage}
				&
				\ \hfill --------------------- \hfill \rule{0pt}{10pt}
				\\ \hline
				%%%%%%%%%%%%%%%%%%%%%%%%%%%%%%%%%%%%%%%%%%%%%%%%%%%%%%%%WL
				%		WL &%\quad
				%		\begin{minipage}{150pt}
					%			$
					%			\begin{cases}
						%				xy^2+y^2x-2yxy, \\
						%				x^2y+yx^2-2xyx
						%			\end{cases}
					%			$ 
					%		\end{minipage}
				%		&
				%		\ \hfill --------------------- \hfill \rule{0pt}{10pt}
				%		\\ \hline
		\end{tabular}}
		\vspace*{1em}
		%%%%%%%%%%%%%%%%%%%%%%%%%%%%%%%%%%%%%%%%%%%%%%%%%%%%%%%%%%%%%%%%%%%%%%%%%%%%%%%%%%%%
		
		For each type, Theorem \ref{thm.main1} and Theorem \ref{thm.main2} are proved by the following
		five steps:
		\begin{description}
			\item[Step 0] Find $\tau_1, \tau_2$ for $E=C_{\tau_{1}} \cup C_{\tau_2}$.
			%\item[Step 1] Find $\Aut^R_k E$.
			\item[Step 1] Find the defining relations of $\cA(E,\sigma)$ for each $\sigma \in \Aut^G_k E$
			by using (G2) condition in Definition \ref{defin.GA}.
			\item[Step 2] Check AS-regularity of $\cA(E,\sigma)$.
			\item[Step 3] Classify them up to isomorphism of graded algebras in terms of their
			defining relations by using Theorems \ref{thm.G1} (1) and \ref{thm.G2} (1).
			\item[Step 4] Classify them up to graded Morita equivalence in terms of their
			defining relations by using Theorems \ref{thm.G1} (2), \ref{thm.G2} (2).
		\end{description}
		
		\textit{Proof of Theorem \ref{thm.main1} and Theorem \ref{thm.main2}.}\
		We will give a proof for Type T$_1$ and T$_2$, that is,
		$E=C_{\tau_1} \cup C_{\tau_2}$ is a union of two irreducible divisors of bidegree $(1,1)$
		in $\PP^1 \times \PP^1$ such that $|C_{\tau_1} \cap C_{\tau_2}|=1$.
		For the other types, the proofs are similar.
		
		\underline{Step 0}: Let  $E=C_{\tau_1} \cup C_{\tau_2}$ be a union of two irreducible divisors of bidegree $(1,1)$ in $\PP^1 \times \PP^1$
		with $|C_{\tau_1} \cap C_{\tau_2}|=1$, and $\sigma \in \Aut_k^G E$.
		By Theorem \ref{thm.Bez}, $\tau_2^{-1}\tau_1 \sim
		\begin{pmatrix}
			1 & 1 \\
			0 & 1
		\end{pmatrix}$.
		Then there exists $P \in \PGL_2(k)$ such that $P^{-1}\tau_2^{-1}\tau_1P=
		\begin{pmatrix}
			1 & 1 \\
			0 & 1
		\end{pmatrix}$. 
		By Lemma \ref{lem.equiv}, we may assume that $\tau_2^{-1}\tau_1=
		\begin{pmatrix}
			1 & 1 \\
			0 & 1
		\end{pmatrix}$.
		Take $(p_0,q_0) \in C_{\tau_1} \cap C_{\tau_2}$. Since $\tau_2^{-1}\tau_1(p_0)=p_0$, we have that $p_0=(1,0)$.
		Since an automorphism of $E$ preserves intersection $C_{\tau_1} \cap C_{\tau_2}$,
		$\sigma(p_0,q_0)=(p_0,q_0)$, so $q_0=p_0=(1,0)$.
		We write $\tau_1=
		\begin{pmatrix}
			\a & \b \\
			\g & \d
		\end{pmatrix}$ where $\a,\b,\g,\d \in k$ with $\a\d-\b\g \neq 0$. Since $(p_0,p_0) \in C_{\tau_1}$, we have that
		\begin{center}
			$p_0=\tau_1(p_0)=(\a,\g)$,
		\end{center}
		so $\a \neq 0$ and $\g=0$. Then we may assume that $\a=1$. We have that $\tau_2=\tau_1
		\begin{pmatrix}
			1 & -1 \\
			0 & 1
		\end{pmatrix}=
		\begin{pmatrix}
			1 & \b-1 \\
			0 & \d
		\end{pmatrix}$.
		
		\underline{Step 1}: Let $A_1=\cA(E,\sigma)$ be a geometric algebra
		of Type T$_1$.
		Since $\sigma$ fixes each complonent, $\sigma$ is uniquely determined as follows;
		\begin{center}
			$\sigma|_{C_{\tau_1}}(p,\tau_1(p))=(\tau_1(p),\tau_1^2(p))$,
			$\sigma|_{C_{\tau_2}}(p,\tau_2(p))=(\tau_2(p),\tau_2^2(p))$.
		\end{center}
		By using (G2) condition, we have that $A_1=\kxy/(f_1,f_2)$;
		\begin{center}
			$\begin{cases}
				f_1=x^2y-\d(\d+1)xyx+\d^3yx^2+(2\b-1)(\d+1)xy^2 \\
				\hfill-(2\b-1)(\d+1)\d yxy+\b(\b-1)(\d+1)y^3, \\
				f_2=xy^2-(\d+1)yxy+\d y^2x
			\end{cases}$
		\end{center}
		where $\b,\d \in k$.
		Let $A_2=\cA(E,\sigma)$ be a geometric algebra
		of Type T$_2$.
		Since $\sigma$ switches each component, $\sigma$ is uniquely determined as follows;
		\begin{center}
			$\sigma|_{C_{\tau_1}}(p,\tau_1(p))=(\tau_1(p),\tau_2\tau_1(p))$,
			$\sigma|_{C_{\tau_2}}(p,\tau_2(p))=(\tau_2(p),\tau_1\tau_2(p))$.
		\end{center}
		By using (G2) condition, we have that $A_2=\kxy/(g_1,g_2)$;
		\begin{center}
			$\begin{cases}
				g_1=x^2y+\d(\d-1)xyx-\d^3yx^2+(2\b-1)\d yxy \\
				\hfill+\b(\b-1)(\d-1)y^3, \\
				g_2=xy^2-(\d-1)yxy-\d y^2x+(2\b-1)y^3
			\end{cases}$
		\end{center}
		where $\b,\d \in k$.
		
		\underline{Step 2}:
		We will give a proof for Type T$_1$.
		For Type T$_2$, the proof is similar.
		Let $A_1=k\langle x,y \rangle/(f_1,f_2)$ be a geometric algebra of Type T$_1$;
		\begin{center}
			$\begin{cases}
				f_1=x^2y-\d(\d+1)xyx+\d^3yx^2+(2\b-1)(\d+1)xy^2 \\
				\hfill-(2\b-1)(\d+1)\d yxy+\b(\b-1)(\d+1)y^3, \\
				f_2=xy^2-(\d+1)yxy+\d y^2x
			\end{cases}$
		\end{center}
		where $\b,\d \in k$.
		If $A_1$ is a $3$-dimensional cubic AS-regular algebra, then
		there exists a twisted superpotential $w$ such that $A_1=\cD(w)$.
		By the calculation, if $w_0:=(ax+cy)f_1+(bx+dy)f_2$
		is a twisted superpotential, then $\d=\pm1$.
		If $\d=-1$, then $A$ is of Type P, so we have that $\d=1$.
		Conversely, we assume that $\d=1$.
		In this case, $A_1=\cD(w)$ where
		\begin{align*}
			w=x^2y^2-2xyxy+xy^2x&+yx^2y-2yxyx+y^2x^2 \\
			&+2(2\b-1)yxy^2-2(2\b-1)y^2xy+2\b(\b-1)y^4.
		\end{align*}
		For any $\b \in k$,
		we take $\theta=
		\begin{pmatrix}
			1 & 2(2\b-1) \\
			0 & 1
		\end{pmatrix} \in \GL_2(k)$. Then we have that $(\tau \otimes \id^{\otimes 3})(\varphi(w))=w$,
		so $w$ is a twisted superpotential.
		By Lemma \ref{lem.tsp}, $A_1$ is standard.
		If we write ${\bf f}={\bf Mx}$ where ${\bf x}=(x,y)^{t}$, ${\bf f}=(f_1,f_2)^{t}$,
		then ${\bf M}=
		\begin{pmatrix}
			-2xy+yx & x^2+2(2\b-1)xy-2(2\b-1)yx+2\b(\b-1)y^2 \\
			y^2 & xy-2yx
		\end{pmatrix}$, so the common zero set of the entries of ${\bf M}$ is empty in $\PP^1 \times \PP^1$.
		By Theorem \ref{thm.ATV}, $A_1=\cD(w)$ is a $3$-dimensional cubic AS-regular algebra.
		
		\underline{Step 3}:
		Let $A=\cA(E,\sigma), A'=\cA(E',\sigma')$ be $3$-dimensional cubic AS-regular algebras of Type T$_1$
		where $E=C_{\tau_{1}} \cup C_{\tau_2}$, $E'=C_{\tau'_{1}} \cup C_{\tau'_2}$,
		$\tau_{1}=
		\begin{pmatrix}
			1 & \b \\
			0 & 1
		\end{pmatrix}$, $\tau_2=
		\begin{pmatrix}
			1 & \b-1 \\
			0 & 1
		\end{pmatrix}$, $\tau'_1=
		\begin{pmatrix}
			1 & \b' \\
			0 & 1
		\end{pmatrix}$, $\tau'_2=
		\begin{pmatrix}
			1 & \b'-1 \\
			0 & 1
		\end{pmatrix}$.
		By Theorem \ref{thm.G1} and Lemma \ref{lem.equiv},
		if $A \cong A'$, then $(\tau'_1,\tau'_2) \sim (\tau_1,\tau_2), (\tau_2,\tau_1)$.
		By the calculation, if $(\tau'_1,\tau'_2) \sim (\tau_1,\tau_2)$, then $\b'=\b$, and
		if $(\tau'_1,\tau'_2) \sim (\tau_2,\tau_1)$, then $\b'=1-\b$.
		Conversely, if $\b'=\b, 1-\b$, then it is clear that $A \cong A'$.
		
		We next show that every $3$-dimensional cubic AS-regular algebra $A=\cA(E,\sigma)$ of Type T$_2$
		is isomorphic to $A'=\cA(E',\sigma')$ where
		$E=C_{\tau_{1}} \cup C_{\tau_2}$, $E'=C_{\tau'_{1}} \cup C_{\tau'_2}$,
		$\tau_{1}=
		\begin{pmatrix}
			1 & \b \\
			0 & -1
		\end{pmatrix}$, $\tau_2=
		\begin{pmatrix}
			1 & \b-1 \\
			0 & -1
		\end{pmatrix}$, $\tau'_1=
		\begin{pmatrix}
			1 & \frac{1}{2} \\
			0 & -1
		\end{pmatrix}$, $\tau'_2=
		\begin{pmatrix}
			1 & -\frac{1}{2} \\
			0 & -1
		\end{pmatrix}$.
		Take
		%\begin{center}
		$\mu:=
		\begin{pmatrix}
			1 & \frac{1}{2}(\b-\frac{1}{2}) \\
			0 & -1
		\end{pmatrix} \in \PGL_2(k)$.
		%\end{center}
		Then we have that $\mu\tau_2=\tau'_1\mu$ and $\mu\tau_1=\tau'_2\mu$,
		so $(\tau_2,\tau_1) \sim (\tau'_1,\tau'_2)$.
		By Lemma \ref{lem.equiv} (1), we have that $E \sim_2 E'$. Moreover, the following diagram
		$$\xymatrix{
			E\ar[r]^-{\mu \times \mu}\ar[d]_-{\sigma}&E'\ar[d]^-{\sigma'}\\
			E\ar[r]_-{\mu \times \mu}&E'
		}$$
		commutes. By Theorem \ref{thm.G2} (1), we have that $A \cong A'$.
		
		\underline{Step 4}:
		Let $A'=\cA(E',\sigma')$ be a $3$-dimensional cubic AS-regular algebra of Type T$_1$
		where $E'=C_{\tau'_{1}} \cup C_{\tau'_2}$,
		$\tau'_{1}=
		\begin{pmatrix}
			1 & \frac{1}{2} \\
			0 & 1
		\end{pmatrix}$, $\tau'_2=
		\begin{pmatrix}
			1 & -\frac{1}{2} \\
			0 & 1
		\end{pmatrix}$.
		We show that every $3$-dimensional cubic-AS-regular algebra of Type T$_1$ is graded Morita
		equivalent to $A'$.
		Let $A=\cA(E,\sigma)$ be a $3$-dimensional cubic AS-regular algebra of Type T$_1$
		where $E=C_{\tau_{1}} \cup C_{\tau_2}$,
		$\tau_1=
		\begin{pmatrix}
			1 & \b \\
			0 & 1
		\end{pmatrix}$, $\tau_2=
		\begin{pmatrix}
			1 & \b-1 \\
			0 & 1
		\end{pmatrix}$.
		For any $i \in \ZZ$, we define an automorphism $\mu_i:\PP^1 \to \PP^1$ by
		\begin{center}
			$\mu_i:=
			\begin{pmatrix}
				1 & i(-\b+\frac{1}{2}) \\
				0 & 1
			\end{pmatrix}$
		\end{center}
		By the calculation, the following diagram
		$$\xymatrix@C=40pt{
			E'\ar@<0.8ex>[r]^-{\mu_i \times \mu_{i+1}}\ar[d]_-{\sigma'}&E\ar[d]^-{\sigma}\\
			E'\ar[r]_-{\mu_{i+1} \times \mu_{i+2}}&E
		}$$
		commutes for every $i \in \ZZ$.
		By Theorem \ref{thm.G2} (2), $\GrMod\,A \cong \GrMod\,B$.
		
		Let $A''=\cA(E'',\sigma'')$ be a $3$-dimensional cubic AS-regular algebra of Type T$_2$
		where $E''=C_{\tau''_{1}} \cup C_{\tau''_2}$,
		$\tau''_{1}=
		\begin{pmatrix}
			1 & \frac{1}{2} \\
			0 & -1
		\end{pmatrix}$, $\tau''_2=
		\begin{pmatrix}
			1 & -\frac{1}{2} \\
			0 & -1
		\end{pmatrix}$.
		Then $A'=\cD(w')$ and $A''=\cD(w'')$ where
		\begin{align*}
			w'=x^2y^2-2xyxy+xy^2x&+yx^2y-2yxyx+y^2x^2 \\
			&+2yxy^2-2y^2xy+2y^4, \\
			w''=x^2y^2+2xyxy+xy^2x&+yx^2y+2yxyx+y^2x^2 \\
			&+2yxy^2-2y^2xy+2y^4.
		\end{align*}
		Take $\theta=
		\begin{pmatrix}
			\sqrt{-1} & 0 \\
			0 & -\sqrt{-1}
		\end{pmatrix} \in \Aut(w'')$. Then we have that $w'=(w'')^{\theta}$, that is, $w'$ is a MS-twist of $w''$ by $\theta$.
		By Lemma \ref{lem.dq}, $A'=\cD(w') \cong \cD(w'')^{\theta}$.
		By Theorem \ref{thm.Zhang}, $\GrMod\,A' \cong \GrMod\,A$.
%%%%%%%%%%%%%%%%%%%%%%%%%%%%%%%%%%%%%%%%%%%%%%%%%%%%%%%%%%%%%%%%%%%%%%%%%%%%%%%%%%%%%%%%%%%%%%%%%%%%%%%%%%%%%%%%%%%%%%%%%%%

\end{document}